\documentclass[final]{siamart}
\usepackage{lipsum}
\usepackage{amsfonts}
\usepackage{graphicx}
\usepackage{algorithmic}
\usepackage{amsmath}
\usepackage{amssymb}
\usepackage{subfig}
\usepackage{booktabs}
\usepackage{multirow}
\usepackage{amsopn}

\graphicspath{{draw/}}
\usepackage{subfig}
\usepackage{color}

\DeclareUnicodeCharacter{00A0}{ }

\newcommand{\bx}{\mathbf{x}}

\newcommand{\md}{\mathrm{d}}

\title{Parallel energy stable phase field simulations of Ni-based alloys system
\thanks This work was supported in part by NSFC (grant \# 11871069 and 12131002), and the Strategic Priority Research Program of the Chinese Academy of Sciences (grant \# XDB22020100).}

\author{Jizu Huang\thanks{LSEC, Academy of Mathematics and Systems Science, Chinese Academy of Sciences, Beijing 100190, China, 
    and School of Mathematical Sciences, University of Chinese Academy of Sciences, Beijing 100049, China  ({\tt huangjz@lsec.cc.ac.cn}).}
  \and Chao Yang\thanks{School of Mathematical Sciences, Peking University, Beijing 100871, China,
 and Institute for Computing and Digital Economy, Peking University, Changsha 410205, China. ({\tt chao\_yang@pku.edu.cn}). }
 Corresponding author: Chao Yang ({\tt chao\_yang@pku.edu.cn}).}

\begin{document}

\maketitle
\begin{abstract}
{In this paper, we investigate numerical methods for solving Nickel-based phase field system related to free energy, including the elastic energy and logarithmic type functionals.  To address the challenge posed by the particular free energy functional, we propose a semi-implicit scheme based on the discrete variational derivative method, which is unconditionally energy stable and maintains the energy dissipation law and the mass conservation law.  Due to the good stability of the semi-implicit scheme, the adaptive time step strategy is adopted, which can flexibly control the time step according to the dynamic evolution of the problem. A domain decomposition based, parallel Newton--Krylov--Schwarz method is introduced to solve the nonlinear algebraic system constructed by the discretization at each time step. Numerical experiments show that the proposed algorithm is energy stable with large time steps, and highly scalable to six thousand processor cores.
}
\end{abstract}



\section{Introduction} 
	Nickel-based (Ni-based) alloys are often used as high-temperature structural materials, such as working blades of aero-engine and turbine disk \cite{bozzolo2012evolution, smith2017diffusion} in the aerospace fields, due to their excellent toughness, high temperature creeps strength, oxidation resistance, corrosion resistance, and anti-fatigue properties.
	The precipitation strengthening of ordered precipitates ($\gamma'$ phase) plays an important role in maintaining the high-temperature mechanical properties of Ni-based alloys, which are embedded in the matrix ($\gamma$ phase) through the heat treatment process.  
	And the properties of Ni-based alloys are closely related to the volume fraction of the ordered precipitates \cite{yadav2016effect}.
	It is therefore important to study the effect of volume fraction on the precipitation kinetics of the $\gamma'$ phase, including the particle size, number density, and spatial distribution \cite{su2018microstructural}.

	The phase field model has been widely used to study microstructure evolutions in Ni-based alloys \cite{liu2018morphology, lu2014phase, wang1998field, wu2016precipitation, zhu2002linking, zhu2004three}, such as phase transitions and dynamic descriptions of the $\gamma'$ phase growth and coarsening.
	In these phase field models, the $\gamma/\gamma'$ microstructure can be described by a composition field variable and several long-range order parameters.
	The total free energy of a Ni-based alloy system, including the local chemical free energy, the elastic energy, and the interfacial energy, is defined as a functional of all field variables. 
	The local chemical free energy is usually constructed by a Landau-type polynomial \cite{wang1998field} or using the CALPHAD method \cite{zhu2002linking}. 
	The elastic energy makes an important contribution to the morphology evolution and the kinetics of growth and coarsening of the $\gamma'$ phase, which is defined by using the micro-elasticity theory proposed in \cite{khachaturyan1995elastic, khachaturyan2013theory}.
	Driven by the total free energy, a gradient flow system for the Ni-based alloys, such as Cahn--Hilliard (CH) and Ginzburg--Landau (GL) coupled equations, can be derived.

	Over the past two decades, a series of efforts \cite{liu2018morphology, lu2014phase, wang1998field, wu2016precipitation, yang2017phase, zhang2020numerical, zhu2002linking, zhu2004three} have been made to numerically simulate the morphological evolution of the Ni-based alloys. 
	Most of these works rely on the semi-implicit Fourier spectral method \cite{chen1998applications, zhu1999coarsening} to solve the CH and GL coupled equations together with the linear elasticity equations in order to maintain stability and achieve high accurate solutions.
	It is well known that there are a plethora of efficient solution algorithms for gradient flow problems, including the convex splitting method \cite{baskaran2013convergence, elliott1993global, eyre1998unconditionally, shen2012second}, the stabilization method \cite{zhu1999coarsening, shen2010numerical}, the exponential time differencing \cite{ju2015fast}, the invariant energy quadratization \cite{yang2016linear, zhao2017numerical}, the scalar auxiliary variable \cite{shen2019new}, the discrete variational derivative (DVD) method \cite{du1991numerical, huang2020parallel, DVDM, PFCDVD}, etc.
	However, due to the existence of the linear elasticity equations and the logarithmic functions in the free energy, it is a non-trivial task to extend the above methods to handle the Ni-based alloys phase field model. 
	On the other hand, the interfacial thickness between the $\gamma$ and $\gamma'$ phase is less than 5 nm \cite{ansara1997thermodynamic} in a typical Ni-based alloy system.
	To maintain numerical accuracy and stability, the mesh size should be smaller than the interfacial thickness, which makes the simulation of Ni-based alloy systems very expensive. 
	As a result, there are only a few works on three dimensional simulations, usually at the scale of tens of nanometers. 
	Such a small computational system is usually far from adequate to study the coarsening process in real three dimensional Ni-based alloys, where a sufficient number of particles is required for proper statistical measurements \cite{zhu2004three}.

	To overcome the above difficulties, we propose an unconditional energy stable semi-implicit scheme based on the DVD method to efficiently solve the Ni-based alloys phase field system, which preserves many properties of the original phase-field system, for instance, the free energy dissipation, the mass conservation, etc. 
	Thanks to the unconditional energy stability of the semi-implicit scheme, an adaptive time stepping strategy \cite{huang2020parallel} can be introduced to dynamically select the time step size according to the accuracy requirement of the solutions and the dynamic features of the system, which can improve the efficiency of the newly proposed scheme. 
	In the semi-implicit method, the most expensive step is to solve a large sparse nonlinear algebraic system at each time step. 
	A parallel, highly scalable, Newton– Krylov–Schwarz (NKS) algorithm \cite{NKS} is applied to solve such a system, where the nonlinear system arisen at each time step is solved by an inexact Newton method. And the Jacobian problem of the NKS algorithm is solved with a preconditioned Krylov subspace method, within which an LU or ILU factorization is used for solving the subdomain problem. 
	To achieve the optimal performance of the NKS algorithm, several key parameters, including the type of the Schwarz preconditioner, the overlap size, and the subdomain solver, are discussed and tested.

	The remainder of this paper is organized as follows. 
	In Sec. 2, we introduce a phase field model for Ni-based alloys and study the mass conservation law and energy dissipation law. 
	In Sec. 3, an unconditional energy stable semi-implicit scheme is proposed based on the DVD method, which is proved to enjoy the mass conservation law and energy dissipation law. 
	In Sec. 4, we introduce the NKS algorithm together with the adaptive time step strategy to solve the nonlinear system. Several three dimensional simulations are reported in Sec. 5 
	and the paper is concluded in Sec. 6.

\section{Phase field model for Ni-based alloys}
	In this section, we introduce a phase field model for the Ni-based alloys to describe the evolution of the $\gamma/\gamma'$ phase driven by total free energy, which includes the Gibbs energy, the interface energy, and the elastic energy.
	The total free energy for the Ni-based alloys has the following form:
\begin{equation}
\label{simple free energy-1}
{\cal E}=\int\limits_\Omega\left[{E}_{\textnormal{G}}(c,\eta)+\frac{ \gamma_c}2| \nabla c|^2+\frac{3\gamma_{\eta}}{2}
|\nabla \eta|^2+E_{\textnormal{el}}(c)\right]\hbox{d} \mathbf x,
\end{equation} 
	where $\Omega\in\mathbb{R}^d$ with $d=1,\,2,$ or $3$ is the computational domain, ${E}_{\textnormal{G}}(c,\eta)$ is the Gibbs energy, and $E_{\textnormal{el}}(c)$ is the elastic energy.
	In \eqref{simple free energy-1}, the conserved concentration field $c:=c(\mathbf x,t)\in(0,0.25)$ denotes the composition of Aluminum (Al) element,
	$\eta:=\eta(\mathbf x,t)\in(0,1)$ is a long-range order parameter field, and $\gamma_c, \,\gamma_\eta$ are 
 	the interface energy coefficients of the composition and long-range order
	parameter, respectively. The Gibbs energy ${E}_{\textnormal{G}}(c,\eta)$ is defined as:
\begin{equation}
{E}_{\textnormal{G}}(c,\eta)=\Phi (c,\eta) + \theta \left\{\Psi (c)+\frac 3 4 \Psi (c\eta)+ \frac 1 4 \Psi (c\eta')\right\},
\end{equation} 
where $\Phi (c,\eta) $ is a polynomial functional of $c$ and $\eta$ given in Appendix A, $\Psi(z)=z\ln z+(1-z)\ln(1-z)$, $\theta $ is a dimensionless parameter, and $\eta'=4-3\eta$.
The elastic strain energy density $E_{\textnormal{el}}(c)$ is given by:
\begin{equation}
\label{elasticenergy}
E_{\textnormal{el}}(c)=\frac 12 \epsilon^{\textnormal{el}}: C:\epsilon^{\textnormal{el}},
\end{equation}
where $C$ is the Hooke's tensor and $\epsilon^{\textnormal{el}}=(\epsilon^{\textnormal{el}}_{IJ})$ is the elastic strain.
The elastic strain can be reformulated as $\epsilon^{\textnormal{el}}=\epsilon-\epsilon^0$,
where $\epsilon$ is the total strain related to the displacement and the eigenstrain $\epsilon^0$ expressed as
$\epsilon^0=\epsilon_0(c-\bar{c}) {\cal I}$.
Here $\epsilon_0$ is a lattice parameter,
${\cal I}$ is the identity matrix, and $\bar{c}=\frac{1}{|\Omega|}\int_\Omega c(\bx,t)\md \bx$ is the average composition.
According to the relationship between the strain and the displacement, the total strain $\epsilon$ can be
expressed as:
$$
\epsilon_{IJ}(\mathbf{x},t)=\frac 1 2\left[\frac{\partial u_I}{\partial x_J}+\frac{\partial u_J}{\partial x_ I}\right],
$$
where $\mathbf{u}(\mathbf{x},t):=(u_1,\cdots,u_d)^T$ denotes the displacement.
Since the mechanical equilibrium for 
the elastic displacement is established much faster than
composition diffusional process, 
the system is always in the mechanical equilibrium
\begin{equation}\label{equilibrium:equ}
\nabla \cdot \sigma =\nabla \cdot \sigma^0,
\end{equation}
where $\sigma=C:\epsilon$ and $\sigma^0=C:\epsilon^0$.

	In the Ni-based alloys, the morphological evolution of the composition field variable $c(\mathbf{x},t)$ and the long-range order parameter $\eta(\mathbf{x},t)$ is described by the CH--GL coupled equations in the following dimensionless form:
	\begin{equation}\label{CHGL}
	\left\{
	\begin{aligned}
	&\frac{\partial c}{\partial t}=\nabla \left[M\nabla \frac {\delta {\cal E}}{\delta c}\right] \quad \hbox{in}~\Omega\times(0,\mathcal {T}],\\
	& \frac{\partial \eta}{\partial t}=-\frac{\delta {\cal E}}{\delta \eta}\quad \hbox{in}~\Omega\times(0,\mathcal {T}],
	\end{aligned}
	\right.
	\end{equation}
	where $M= \varkappa c(1-c) $ with $ \varkappa$ being a mobility parameter.
	The variational derivatives of the free energy functional $ \frac {\delta {\cal E}}{\delta c}$ and $ \frac {\delta {\cal E}}{\delta \eta}$ in \eqref{CHGL} are defined as: 
	\begin{equation}\label{variational derivatives}
	\left\{
	\begin{aligned}
	&\frac {\delta {\cal E}}{\delta c}=\frac{\partial E_{\textnormal G}(c,\eta)}{\partial c}-\gamma_c\Delta c+\frac{\partial E_{\textnormal{el}}(c)}{\partial c},\\
	& \frac {\delta {\cal E}}{\delta \eta}=\frac{\partial E_{\textnormal G}(c,\eta)}{\partial \eta}-3\gamma_\eta\Delta \eta.
	\end{aligned}
	\right.
	\end{equation}
	More details about the Ni-based alloys phase field system are discussed in Appendix A.
	The initial conditions for CH--GL coupled equations \cref{CHGL} are given as $c^0=c(\bx,0)$ and $\eta^0=\eta(\bx,0)$.
	We consider periodic boundary conditions or the following homogeneous Neumann boundary conditions
\begin{equation}
\label{boundaryequ}
\mathbf{n}\cdot \nabla c|_{\partial \Omega}=\mathbf{n}\cdot \nabla \eta|_{\partial \Omega}
=\mathbf{n}\cdot M\nabla \frac{\delta {\cal E}}{\delta c}\Big|_{\partial \Omega}=
 \mathbf{n}\cdot {\sigma^{\textnormal{el}}}|_{\partial \Omega}=0,
\end{equation}
where $\sigma^{\textnormal{el}}=C:\epsilon^{\textnormal{el}}$. 

Several important properties of the phase field model for Ni-based alloys are described in the following theorem.

\begin{theorem}
The Ni-based alloys phase field model given by \cref{CHGL} and \cref{equilibrium:equ} together with periodic boundary conditions or homogeneous Neumann boundary conditions \cref{boundaryequ} satisfies the following properties.
\begin{itemize}
\item[1.] Mass conservation law: $\int_\Omega c(\bx,t)\md \bx=\int_\Omega c(\bx,0)\md \bx$.
\item[2.] Energy dissipation law: $\frac{\md}{\md t}{\cal E} \leq 0$.
\item[3.] The mechanical equilibrium equation: $\frac{\delta{\cal E}}{\delta\mathbf{u}}=0$.
\end{itemize}
\end{theorem}

\begin{proof}
 Integrating the first equation of \cref{CHGL} over domain $\Omega$, we obtain the conservation law $\int_\Omega c(\bx,t)\md \bx=\int_\Omega c(\bx,0)\md \bx$ immediately. 
 According to \cref{elasticenergy} and \cref{equilibrium:equ}, we have 
 \begin{equation}
 \frac{\delta{\cal E}}{\delta\mathbf{u}}=-\nabla\cdot \frac{\partial E_{\textnormal{el}}}{\partial(\nabla\mathbf{u})}=
 -\nabla\cdot \frac{\partial E_{\textnormal{el}}}{\partial \epsilon}=-\nabla \cdot \sigma^{\textnormal{el}}= 0.
 \end{equation}

The free energy functional of the Ni-based alloys phase field system satisfies the following equation  
\begin{equation}
\label{DIS1}
\begin{aligned}
 \frac{\md}{\md t}{\cal E} &= \frac{\md}{\md t}\int\limits_\Omega\left\{E_{\textnormal{G}}(c,\eta)+\frac{ \gamma_c}2| \nabla c|^2+\frac{3\gamma_{\eta}}{2}
|\nabla \eta|^2+E_{\textnormal{el}}(c)\right\}\hbox{d} \mathbf x \\
 &= \int\limits_\Omega \Big\{\frac{\partial E_{\textnormal{G}}}{\partial c}\frac{\partial c}{\partial t}+\gamma_c\nabla c\cdot\frac{\partial \nabla c}{\partial t}+\frac{\partial E_{\textnormal{el}}}{\partial c}\frac{\partial c}{\partial t}
 +\frac{\partial E_{\textnormal{el}}}{\partial(\nabla\mathbf{u})}\cdot \frac{\partial (\nabla \mathbf{u})}{\partial t} \\
 &\qquad +\frac{\partial E_{\textnormal{G}}}{\partial \eta}\frac{\partial \eta}{\partial t}+3\gamma_\eta\nabla \eta\cdot \frac{\partial \nabla \eta}{\partial t}\Big\}\md\bx\\
 &= \int\limits_\Omega\left\{ \frac {\delta {\cal E}}{\delta c}\frac{\partial c}{\partial t}+\frac {\delta {\cal E}}{\delta \eta} \frac{\partial \eta}{\partial t}\right\}\md\bx + B_1
 \\&= -\int\limits_\Omega\left\{ M\nabla \frac {\delta {\cal E}}{\delta c}\cdot\nabla\frac {\delta {\cal E}}{\delta c}+\left(\frac{\delta {\cal E}}{\delta \eta}\right)^2\right\}\md\bx + B_2,
\end{aligned}
\end{equation}
where $B_1$ and $B_2$ are the boundary terms due to integration-by-parts, which vanishes with the given boundary conditions. 
Thus we have
\begin{equation}
\frac{\md}{\md t}{\cal E} \leq 0,
\end{equation}
which demonstrates the energy dissipation of the Ni-based alloys phase field system. 
\end{proof}

By denoting ${\cal E}_\delta:={\cal E}(c+\delta c,\eta+\delta\eta,\mathbf{u}+\delta\mathbf{u})$, 
the following integral relationship between the variational derivative and the free energy holds
\begin{equation}
\label{condG}
\begin{aligned}
{\cal E}_\delta- {\cal E}(c,\eta,\mathbf{u})&= 
 \int\limits_\Omega \Big\{\frac{\partial E_{\textnormal{G}}}{\partial c}\delta c+\gamma_c\nabla c\cdot\nabla \delta c+\frac{\partial E_{\textnormal{el}}}{\partial c}\delta c
 +\frac{\partial E_{\textnormal{el}}}{\partial (\nabla \mathbf{u})}\cdot\nabla\delta \mathbf{u}\\
 &\qquad +\frac{\partial E_{\textnormal{G}}}{\partial \eta}\delta \eta+3\gamma_\eta\nabla \eta \cdot\nabla \delta \eta\Big\}\md\bx\\
 & \approx \int\limits_\Omega\left\{ \frac {\delta {\cal E}}{\delta c}\delta c+\frac {\delta {\cal E}}{\delta \eta}\delta \eta\right\}\md\bx + B.
\end{aligned}
\end{equation}
The first equality is obtained by using Taylor expansion, and the second equality comes from 
the integration-by-parts formula and $\frac{\delta{\cal E}}{\delta \mathbf{u}}=0$. Here $B$ represents the boundary terms from the integration-by-parts formula
and equals zero due to the given boundary conditions. Equation~\eqref{condG} shows the connection between the variational derivative 
and the total free energy functional, which plays an important role in the construction of the energy-stable numerical scheme for the Ni-based alloys phase field system.

\section{Discretization for the Ni-based alloys phase field model}
We rewrite the Ni-based alloys phase field model given by \cref{CHGL} and \cref{equilibrium:equ}  in a vector form
\begin{equation}
\label{system-1}\left\{
\begin{aligned}
&\frac{\partial \mathbf{V}}{\partial t} = -{\cal A}\mathbf{G},\\
&\nabla \cdot \sigma =\nabla \cdot \sigma^0,
\end{aligned}
\right.
\end{equation}
where
$\mathbf{V}=(c,\,\eta)^T$,  ${\cal A}=\textnormal{diag}(-\nabla \cdot M \nabla,1)$ with the mobility $M=\varkappa c(1-c)$,
and $\mathbf{G}:=\frac{\delta {\cal E}}{\delta \mathbf{V}}$ represents the variational derivative of the free 
energy functional ${\cal E}$ with respect to $\mathbf{V}$. 
The local free energy is decomposed into four parts, which are  
\begin{equation}
\label{neq:01}
\begin{aligned}
& E_1 (\mathbf{V})= \Phi(c,\eta),\\
& E_2(\mathbf{V})= \theta \left\{\Psi (c)+\frac 3 4 \Psi (c\eta)+ \frac 1 4 \Psi (c\eta')\right\},\\
&E_3(\mathbf{V},\nabla\mathbf{u})= E_{\textnormal{el}}(c),\\
& E_4 (\nabla \mathbf{V}^T)=\frac{ \gamma_c}2| \nabla c|^2+\frac{3\gamma_{\eta}}{2}
|\nabla \eta|^2.
\end{aligned}
\end{equation}
It is easy to check that ${\cal A}$ is a semi-positive operator for $c\in(0, \ 0.25)$ and $\eta\in [0,\ 1]$.

Without loss of generality, we consider discretizing the Ni-based alloys phase field system on a one dimensional domain $\Omega=[0,L_x]$, which is covered by a uniform mesh with the mesh size $\Delta x = L_x/N_x$.
The temporal interval $[0, {\cal T}]$ is split by a set of nonuniform points $\{t_n\}_0^{N_{\cal T}}$ with the $n$-th time step size $\Delta t_n=t_{n+1}-t_n$. 
Let us denote  $\tilde{\mathbf{V}}^n_{i}\approx \mathbf{V}(x_i, t_n)$ and $ \tilde{{u}}^n_{i}\approx {u}(x_i, t_n)$ as the approximate solutions of the Ni-based alloys phase field  system
at $t_n$, where
$x_i= (i-\frac{1}{2})\Delta x$ 
with $1\leq i \leq N_x$. 
Throughout the paper, notations with a tilde are the corresponding approximate solutions or functions at the discrete level.
We denote  $[\tilde\phi]_{i}^{n+\frac 12}:=\frac 1 2 ([\tilde\phi]_{i}^n+[\tilde\phi]_{i}^{n+1})$ as the approximation of any function $\phi$ at $t_{n+\frac 12}$.
 Let us introduce some useful notations as following:
$$[D_x^+\tilde\phi]^n_{i}=\frac{\tilde\phi^n_{i+1}-\tilde\phi^n_{i}}{\Delta x},\quad \ [D_x^-\tilde\phi]^n_{i}=\frac{\tilde\phi^n_{i}-\tilde\phi^n_{i-1}}{\Delta x},\quad \
[D_x \tilde\phi]^n_{i}=\frac{\tilde\phi^n_{i+\frac{1}{2}}-\tilde\phi^n_{i-\frac{1}{2}}}{\Delta x}, 
$$
 $$[\hat {D}_x \tilde\phi]^n_{i}=\frac{\tilde\phi^n_{i+1}-\tilde\phi^n_{i-1}}{2\Delta x}, \quad \ 
 \left[(D_x^\pm\tilde\phi)^2\right]^n_{i}=\frac{\left([D_x^+\tilde\phi]^n_{i}\right)^2+\left([D_x^-\tilde\phi]^n_{i}\right)^2}{2}.$$
The operator $\nabla \cdot M \nabla$ in $\cal A$ is discretized by
$[(\nabla \cdot  \tilde{M}\nabla)\tilde\phi]_{i}^n =  [D_x  \tilde M D_x \tilde\phi]_{i}^n,$
where the
approximate value of the mobility $M$ is calculated as
$$\tilde{M}_{i+\frac 12} =\varkappa \frac{\tilde c^n_{i }(1-\tilde c^n_{i })+\tilde c^n_{i+ 1}(1-\tilde c^n_{i+ 1})}2.$$
The discrete operator $\tilde{{\cal A}}$ for the Ni-based alloy phase field system is then defined as $\tilde{{\cal A}}(\tilde\varphi_{i}^n,\tilde\phi_{i}^n)^T
=(-[\nabla \cdot \tilde{M}\nabla\tilde\varphi]_{i}^n,\tilde\phi_{i}^n)^T$.
With the periodic boundary conditions or the homogeneous Neumann boundary conditions, 
we present  vital formula
\begin{equation}
\label{summation}
\begin{aligned}
-\sum_{i=1}^{N_x}\tilde \varphi^n_{i}\left[(D_x \tilde M D_x)\tilde \phi\right]^{n}_{i}\color{black}
=\frac 1 2
\sum_{i=1}^{N_x}
\left(\tilde M^n_{i+\frac 12}[D_x^+ \tilde \varphi]^{n}_{i} [D_x^+\tilde \phi]^{n}_{i}\right).
   \end{aligned}
\end{equation}
\color{black}

\color{black}
It follows from the Cauchy--Schwarz inequality  and \eqref{summation} that
\color{black}
\begin{equation}
\label{neq:02}
\left\{
\begin{aligned}
&\left<\tilde \phi^n_{i}[{A}_{11}\tilde\phi]_{i}^n\right>\geq 0,\\
&\left<\tilde \varphi^n_{i}[{A}_{11}\tilde\phi]_{i}^n\right>
=\left<\tilde \phi^n_{i}[{A}_{11}\tilde\varphi]_{i}^n\right>,
\\
&\left<\tilde \varphi^n_{i}[{A}_{11}\tilde\phi]_{i}^n\right>
\leq \left< \tilde \phi^n_{i}[{A}_{11}\tilde\phi]_{i}^n\right>^{1/2}\left< 
\tilde \varphi^n_{i}[{A}_{11} \tilde\varphi]_{i}^n\right>^{1/2},
\end{aligned}
\right.
\end{equation}
where $A_{11}=-\nabla \cdot M \nabla$ 
and $\left<\Box_{i} \right>$ is defined as $\left<\Box_i\right>:=\sum_{i=1}^{N_x}\Box_{i} \Delta x $.
Thus, 
we conclude that the discrete operator $\tilde{\cal A}$ for the Ni-based alloys phase field system is semi-positive.

With the aforementioned notations, we construct a semi-implicit scheme for the Ni-based alloys phase field system \cref{system-1} by the following steps.
In the first step,
the discretizations of the local free energies $E_i$ (i=1,~2,~3,~4) and the total free energy ${\cal E}$ at time $t_n$ are respectively defined as
\begin{equation}
\label{relations}
\left\{
\begin{aligned}
\left[\tilde{E}_1\right]_{i}^{n}= &E_1(\tilde c_{i}^n,\tilde \eta_{i}^n),\\ \left[\tilde{E}_2\right]_{i}^{n}=& E_2(\tilde c_{i}^n,\tilde \eta_{i}^n),\\
\left[\tilde{E}_3\right]_{i}^{n}=&\frac 1{2} \left[\tilde\epsilon^{\textnormal{el}}\right]_{i}^{n}:C:\left[\tilde\epsilon^{\textnormal{el}}\right]_{i}^{n},\\
\left[\tilde{E}_4\right]_{i}^{n}=&\frac{\gamma_c}{2}\left[(D_x^\pm\tilde{c})^2\right]^n_{i}
         +\frac{3\gamma_\eta}{2}\left[(D_x^\pm\tilde{\eta})^2\right]^n_{i},
   \end{aligned}
   \right.
\end{equation}
and
\begin{equation}
\label{freeenergy}
\begin{aligned}
\tilde{{\cal E}}^{n} :=\sum_{i=1}^{N_x}\left([\tilde{E}_1]^n_{i}+[\tilde{E}_2]^n_{i}+[\tilde{E}_3]^n_{i}+[\tilde{E}_4]^n_{i}\right)\Delta x.
\end{aligned}
\end{equation}

In \cref{relations}, the discrete elastic strain $\left[\tilde\epsilon^{\textnormal{el}}\right]_{i}^{n}$ is defined as
$$\left[\tilde\epsilon^{\textnormal{el}}\right]_{i}^{n}=\left[{\tilde\epsilon}\right]^n_{i}-\epsilon_0\left(\tilde c^n_{i}-\bar{c}^0\right),$$
where $\bar{c}^0=\frac{1}{|\Omega|}\int_\Omega c(x,0) \md x$ and the  heterogeneous strain $\left[\tilde{\epsilon}\right]^n_{i}$ is discretized as
$$\left[\tilde{\epsilon}\right]^n_{i}=[\hat{D}_x \tilde u]^n_{i}.$$
The discrete elastic strain energy is obtained by solving the following discretization system for the mechanical equilibrium equation
\begin{equation}
\label{mechanicalequilibrium-dis}
\left[\hat D_x\tilde\sigma^{\textnormal{el}}\right]_{i}^{n}=0,
\end{equation}
where $\left[ \tilde\sigma^{\textnormal{el}}\right]_i^{n}=C\left[\tilde\epsilon^{\textnormal{el}}\right]_{i}^{n}$.

In the second step,  a discrete form of the variational derivative ${\mathbf{G}}$ is constructed. 
By taking $\mathbf{V}(x_i, t_n) := \tilde{\mathbf{V}}^{n}_{i}$ and $\delta \mathbf{V}\big|_{x_i}:= \tilde{\mathbf{V}}^{n+1}_{i} - \tilde{\mathbf{V}}^{n}_{i}$,  
we  choose the discrete variational derivative  $\tilde{\mathbf{G}}_{i} $ such that the following summation formula exactly holds
\begin{eqnarray}
\label{DVD}
\begin{aligned}
\tilde{{\cal E}}^{n+1}-\tilde{{\cal E}}^{n}&
=\left<  \tilde{\mathbf{G}}_{i}  
\cdot \left(\tilde{\mathbf{V}}_{i}^{n+1}-\tilde{\mathbf{V}}_{i}^{n}\right) \right>.
\end{aligned}
\end{eqnarray}
 Equation (\ref{DVD})
can be viewed as a discrete form of  (\ref{condG}).
According to \eqref{neq:01} and \eqref{DVD}, we can derive the discrete variational derivative
to the following formula
\begin{equation*}
\tilde{\mathbf{G}} _{i} :=\sum_{K=1}^4 \tilde{\mathbf{G}}^K _{i} = \sum_{K=1}^3\mathbf{G}_K(\tilde{\mathbf{V}}_{i}^{n+1}, 
\tilde{\mathbf{V}}_{i}^{n}) + \mathbf{G}_4([\nabla \tilde{\mathbf{V}}^T]_{i}^{n+1}, 
[\nabla \tilde{\mathbf{V}}^T]_{i}^{n}) ,
\end{equation*}
where $\tilde{\mathbf{G}}^K_{i}$ are derived from $E_K$ with $K=1,\ 2,\ 3,$ and  $4$ in local free energy, respectively.  
Following the framework proposed in \cite{huang2020parallel}, we have that $\tilde{\mathbf{G}}^1_{i}$ is also in polynomial form. 
The cumbersome definition of $\tilde{\mathbf{G}}^1_{i}$ is not given  here but instead in Appendix B. 
$\tilde{\mathbf{G}}^3_{i}$ and $\tilde{\mathbf{G}}^4_{i}$  are respectively defined  as follows
\begin{equation}
\tilde{\mathbf{G}}^3_{i}:=
\left(
\begin{array}{l}
\begin{aligned}
-\epsilon_{0}C[\tilde\epsilon^{\textnormal{el}}]^{n+\frac 1 2}_{i}
  \\ 
0\qquad\qquad
    \end{aligned}
\end{array}\right),
\end{equation}
and
\begin{equation}
\tilde{\mathbf{G}}^4_{i}:=
\left(
\begin{array}{l}
\begin{aligned}
  -{\gamma_c} \left[D_x^2\tilde{c}\right]_{i}^{n+\frac 1 2}
  \\ 
  -{3\gamma_\eta}\left[D_x^2\tilde{\eta}\right]_{i}^{n+\frac 12}
    \end{aligned}
\end{array}\right).
\label{dvds-1}
\end{equation}

However,  since $E_2$ is not a polynomial of $\mathbf{V}$, there does not exist a $\tilde{\mathbf{G}}^2_{i}$  with polynomial form such that
\eqref{DVD} exactly holds.
Consider an alternative form of
$\tilde{\mathbf{G}}^2_{i}$ as 
\begin{equation}\label{D2}
    \tilde{\mathbf{G}}^2_{i}:=\frac\theta 4 \left(
 \begin{aligned}
  4\Psi[\tilde{c}_{i}^{n+1},\tilde{c}_{i}^{n}] +3 \eta^{n+\frac 1 2}_{i}\Psi[p_{i}^{n+1},p_{i}^{n}]
  +[\tilde\eta']^{n+\frac 1 2}_{i}\Psi[q_{i}^{n+1},q_{i}^{n}]\\
    3\tilde c^{n+\frac 1 2}_{i}\left(\Psi[p_{i}^{n+1},p_{i}^{n}]- \Psi[q_{i}^{n+1},q_{i}^{n}]
      \right) \qquad\qquad\end{aligned} \right)
\end{equation}
such that $\tilde{\mathbf{G}}^2_{i}\cdot\left(\tilde{\mathbf{V}}_{i}^{n+1}-\tilde{\mathbf{V}}_{i}^{n}\right)=[\tilde{E}_2]_{i}^{n+1}
-[\tilde{E}_2]_{i}^{n}$. 
Here $p_{i}^{n}=\tilde{c}_{i}^n\tilde{\eta}_{i}^n$, $q_{i}^n=\tilde{c}_{i}^n(4-3\tilde{\eta}_{i}^n)$,
and $\Psi[y_1,y_2]:=(\Psi(y_1)-\Psi(y_2))/(y_1-y_2)$ is the first order finite quotient of function $\Psi$.
Using the trapezoidal rule at the half-time level, 
we obtain a fully discretized scheme for the Ni-based alloys phase field system as:
\begin{equation}\left\{
\begin{aligned}
&\frac{\tilde{\mathbf{V}}_{i}^{n+1} - \tilde{\mathbf{V}}_{i}^{n}}{\Delta t_n} = -
\tilde{{\cal A}} \tilde{\mathbf{G}}_{i},\\
&\left[\hat D_x\tilde\sigma^{\textnormal{el}}\right]_{i}^{n+1}=0.
\end{aligned}
\right.
\label{scheme-1}
\end{equation} 
The stability of the proposed scheme  \eqref{scheme-1} is given by the following theorem.
\color{black}
\begin{theorem}\label{theorem:1}
For any given time step $\Delta t_n >0$, the numerical scheme \eqref{scheme-1} 
is unconditionally energy stable and  satisfies the following properties.
\begin{itemize}
\item[1.] Mass conservation law: $\left<\tilde{c}^{n+1}_{i}\right>=\left<\tilde{c}^{n}_{i}\right>$.
\item[2.] Energy dissipation law: $\tilde {\cal E}^{n+1}\leq \tilde {\cal E}^n$.
\end{itemize}
\end{theorem}
\color{black}
\begin{proof}
According to   \eqref{summation} and \eqref{scheme-1}, by setting $ \varphi=1$ and $ \phi =\frac{\delta {\cal E}}{\delta c}$,
we have the following mass conservation law
\begin{equation}
\left<\tilde{c}^{n+1}_{i}\right>-\left<\tilde{c}^{n}_{i}\right>=\Delta t_n \left<\left[(D_x \tilde M D_x)\tilde \phi\right]^{n}_{i}
\right>=0.
\end{equation}
To prove the energy dissipation law,  we shall show that the summation formula \eqref{DVD} is exactly held with the definitions of discrete variational derivatives $\tilde{\mathbf{G}}_{i}$.
Similar to the proof of \cref{summation}, we obtain the following equation from the discrete mechanical equilibrium scheme \cref{mechanicalequilibrium-dis}
\begin{equation}
 \left<[\tilde\epsilon]^{n}_{i}[\tilde\epsilon^{\textnormal{el}}]^{n+1}_{i}\right>=
  \left<[\tilde\epsilon]^{n+1}_{i}[\tilde\epsilon^{\textnormal{el}}]^{n+1}_{i}\right>=0.
\end{equation}
Then, we get the following formula
\begin{equation}
\label{localsum1}
\begin{aligned}
&\left<\left[\tilde E_{\textnormal{el}}\right]_{i}^{n+1}-\left[\tilde E_{\textnormal{el}}\right]_{i}^{n}\right>\\
&\qquad=-C\left<\left[\tilde\epsilon^{\textnormal{el}}\right]_{i}^{n+\frac 12}(\left[\tilde\epsilon^{0}\right]_{i}^{n+1}-\left[\tilde\epsilon^{0}\right]_{i}^{n})\right>
+C\left<(\left[\tilde \epsilon\right]_{i}^{n+1}-\left[\tilde \epsilon\right]_{i}^{n})\left[\tilde\epsilon^{\textnormal{el}}\right]_{i}^{n+\frac 12}\right>
\\&\qquad=-\left<\epsilon_0 C\left[\tilde\epsilon^{\textnormal{el}}\right]_{i}^{n+\frac 12}(\tilde c_{i}^{n+1}-\tilde c_{i}^{n})\right>
=\left<\tilde{\mathbf{G}}^3_{i}\cdot \left({\tilde{\mathbf{V}}_{i}^{n+1} - \tilde{\mathbf{V}}_{i}^{n}}\right)\right>.
\end{aligned}
\end{equation}
Using the definitions of $\tilde{\mathbf{G}}^1_{i}$, $\tilde{\mathbf{G}}^2_{i}$, $\tilde{\mathbf{G}}^4_{i}$ and  \cref{summation},
we have
\begin{equation}\label{localsum2}
\left\{
\begin{aligned}
&\left<\left[\tilde E_1\right]_{i}^{n+1}-\left[\tilde E_1\right]_{i}^{n}\right>=\left<\tilde{\mathbf{G}}^1_{i}\cdot \left({\tilde{\mathbf{V}}_{i}^{n+1} - \tilde{\mathbf{V}}_{i}^{n}}\right)\right>,\\
&\left<\left[\tilde E_2\right]_{i}^{n+1}-\left[\tilde E_2\right]_{i}^{n}\right>=\left<\tilde{\mathbf{G}}^2_{i}\cdot \left({\tilde{\mathbf{V}}_{i}^{n+1} - \tilde{\mathbf{V}}_{i}^{n}}\right)\right>,\\
&\left<\left[\tilde E_4\right]_{i}^{n+1}-\left[\tilde E_4\right]_{i}^{n}\right>=\left<\tilde{\mathbf{G}}^4_{i}\cdot \left({\tilde{\mathbf{V}}_{i}^{n+1} - \tilde{\mathbf{V}}_{i}^{n}}\right)\right>.\\\end{aligned}\right.
\end{equation}
By combining \eqref{neq:02}, \eqref{DVD},  \eqref{scheme-1},  \cref{localsum1}, and \cref{localsum2},
the discrete free energy induced by scheme \eqref{scheme-1} satisfies
\begin{equation}
\begin{aligned}
   \frac{\tilde {\cal E}^{n+1}- \tilde {\cal E}^n}{\Delta t_n}
   =\left<
 \tilde{\mathbf{G}}_{i}  \cdot \frac{\tilde{\mathbf{V}}_{i}^{n+1} - \tilde{\mathbf{V}}_{i}^{n}}{\Delta t_n}\right>=-\left<
 \tilde{\mathbf{G}}_{i}  \cdot 
\tilde{{\cal A}} \tilde{\mathbf{G}}_{i}  \right>\leq 0. 
   \end{aligned}
   \label{DF}
\end{equation}
\end{proof}
\color{black}

\color{black}
Since the total free energy functional $\cal E$ is non-convex, the existence and uniqueness of the solution  for system \eqref{scheme-1} (especially for large values of $\Delta t$) are not immediate. 
Nonetheless, the proof of Theorem 2.1 does not require a unique solution to \eqref{scheme-1}.
Even if one could prove that scheme \eqref{scheme-1} is  unconditionally energy stable, it is numerically unstable when applied  to  solve the Ni-based alloys phase field system.
The reason is that the numerical computation  of the first order finite quotient  $\Psi[y_1,y_2]$ is unstable and inaccurate when $y_2-y_1$ is close to zero. 
Unfortunately, in  the Ni-based alloys phase field system, $\delta p_{i}=p_{i}^{n+1}-p_{i}^n$  or $\delta q_i =q_{i}^{n+1}-q_{i}^n$ is often close to zero,  which indicates that the numerical calculation of $\tilde{\mathbf{G}}^2_{i}$ by \eqref{D2} is numerically unstable and inaccurate in this situation. 
As a result, the scheme \eqref{scheme-1} may lead to an inaccurate or non-physical solution with the existence of the complicated function in $E_2$.

To overcome the difficulty, we propose a stable and highly accurate approximation for $\tilde{\mathbf{G}}^2_{i}$ in the third step, 
which was first introduced by us in \cite{huang2020parallel} for solving a coupled Allen--Cahn and Cahn--Hilliard system. 
Let us assume $\Psi_S(y,\delta)=\Psi(y)+\sum_{s=1}^{2S} \frac{\delta^s \textnormal{d}^s\Psi(y)}{s! (\textnormal d y)^s}$,
which is a polynomial of  the small parameter $\delta$. The first order quotient $\Psi[y_1,y_2]$
can be approximately computed by 
$$\Psi[y_1,y_2]\approx \Psi_S[y_1,y_2]=\frac{\Psi_S((y_1+y_2)/2,\delta)-\Psi_S((y_1+y_2)/2,-\delta)}{2\delta},$$
which is a polynomial of the small parameter $\delta$ with $\delta = (y_2-y_1)/2$.
Then,
$ \tilde{\mathbf{G}}^2_{i}$ is approximated as
\begin{equation}\label{D2-1}   
\begin{aligned}
 \tilde{\mathbf{G}}^{2,S}_{i}:&=\frac\theta 4 \left(
 \begin{aligned}
  4\Psi_S[\tilde{c}_{i}^{n+1},\tilde{c}_{i}^{n}] +3 \eta^{n+\frac 1 2}_{i}\Psi_S[p_{i}^{n+1},p_{i}^{n}]
  +[\tilde\eta']^{n+\frac 1 2}_{i}\Psi_S[q_{i}^{n+1},q_{i}^{n}]\\
    3\tilde c^{n+\frac 1 2}_{i}\left(\Psi_S[p_{i}^{n+1},p_{i}^{n}]- \Psi_S[q_{i}^{n+1},q_{i}^{n}]\right)\qquad\qquad
       \end{aligned} \right)\\
       &=\tilde{\mathbf{G}}^2_{i}-{\mathbf{R}}^{S}_{i}.
       \end{aligned}
\end{equation}
Based on the Taylor expansion of $\Psi$, we have the truncation error ${{\mathbf R}}^S_{i}\rightarrow 0$ as $S\rightarrow\infty$.

With the approximation, 
the scheme \eqref{scheme-1} for the Ni-bases phase field system is replaced with
\begin{equation}
\label{NSOA-1}
\left\{
\begin{aligned}
&\frac{\tilde{\mathbf{V}}_{i}^{n+1} - \tilde{\mathbf{V}}_{i}^{n}}{\Delta t_n} = -
\tilde{{\cal A}} \tilde{\mathbf{G}}_{i}^S ,\\
&\left[\hat D_x\tilde\sigma^{\textnormal{el}}\right]_{i}^{n+1}=0,
\end{aligned}
\right.
\end{equation}
where the approximate discrete variational derivative $\tilde{\mathbf{G}}_{i}^S $ is defined as
\begin{equation}
\tilde{\mathbf{G}}_{i}^S :=\tilde{\mathbf{G}}^1_{i}+\tilde{\mathbf{G}}_{i}^{2,S}+\tilde{\mathbf{G}}_{i}^3 + \tilde{\mathbf{G}}_{i}^4.
\end{equation}
It should be noted that the solution obtained by using scheme \eqref{NSOA-1} is usually different from the one by scheme \eqref{scheme-1} due to 
the different approximations made by the two schemes.
For  simplicity, we still denote  the solution  of scheme \eqref{NSOA-1} as $\tilde{\mathbf{V}}_{i}^{n+1}$.
According to \eqref{DVD} and \eqref{D2-1}, we have 
\begin{eqnarray}
\label{eq:19}\qquad\quad
\tilde{{\cal E}}^{n+1}-\tilde{{\cal E}}^{n}=\left<  \tilde{\mathbf{G}}_{i}  
\cdot \left(\tilde{\mathbf{V}}_{i}^{n+1}-\tilde{\mathbf{V}}_{i}^{n}\right) \right>
=\left< \left( \tilde{\mathbf{G}}_{i}^{S} +{\mathbf R}^S_{i}\right) 
\cdot \left(\tilde{\mathbf{V}}_{i}^{n+1}-\tilde{\mathbf{V}}_{i}^{n}\right) \right>.~~~~
\end{eqnarray}
The stability of the proposed scheme  \eqref{NSOA-1} is given by the following theorem.

\begin{theorem}\label{theorem2}
Given 
 $\Delta t_n>0$, there exists an integer $S_0$ such that  the scheme~\eqref{NSOA-1} is unconditionally energy stable and satisfies the following properties.
 \begin{itemize}
\item[1.] Mass conservation law: $\left<\tilde{c}^{n+1}_{i}\right>=\left<\tilde{c}^{n}_{i}\right>$.
\item[2.] Energy dissipation law: $\tilde {\cal E}^{n+1}\leq \tilde {\cal E}^n$ for any $S>S_0$.
\end{itemize}
\end{theorem}
\begin{proof}
Similar to \cref{theorem:1}, one can complete the proof of mass conservation law. 
According to  \eqref{neq:02}, \eqref{NSOA-1}, and \eqref{eq:19},
the discrete free energy decided by scheme \eqref{NSOA-1} satisfies
\begin{equation}\label{eq::27}
\begin{aligned}
   &\frac{\tilde {\cal E}^{n+1}- \tilde {\cal E}^n}{\Delta t_n}
   = \left<
 \tilde{\mathbf{G}}_{i} \cdot \frac{\tilde{\mathbf{V}}_{i}^{n+1} - \tilde{\mathbf{V}}_{i}^{n}}{\Delta t_n} \right>=-\left<
\tilde{\mathbf{G}}_{i} \cdot 
\tilde{{\cal A}} \tilde{\mathbf{G}}^S_{i}\right>\\
&\qquad=-
\left<
\tilde{\mathbf{G}}_{i}\cdot 
\tilde{{\cal A}} \tilde{\mathbf{G}}_{i} -
\tilde{\mathbf{G}}_{i} \cdot 
\tilde{{\cal A}} {\mathbf R}_{i}^S\right> \\
&\qquad \leq-\left<
\tilde{\mathbf{G}}_{i}  \cdot 
\tilde{{\cal A}}\tilde{\mathbf{G}}_{i}\right> +\left<
\tilde{\mathbf{G}}_{i}  \cdot 
\tilde{{\cal A}} \tilde{\mathbf{G}}_{i}\right>^{1/2}\left<
{\mathbf R}^S_{i} \cdot \tilde{{\cal A}} {\mathbf R}^S_{i} \right>^{1/2}.
   \end{aligned}
\end{equation}
If ~$\left<
\tilde{\mathbf{G}}_{i}  \cdot 
\tilde{{\cal A}} \tilde{\mathbf{G}}_{i}\right>=0$,
 we have $\tilde {\cal E}^{n+1}\leq \tilde {\cal E}^n$. Otherwise, we set
 $\left<
\tilde{\mathbf{G}}_{i}  \cdot 
\tilde{{\cal A}} \tilde{\mathbf{G}}_{i}\right>=\varrho>0$.
Since ${\mathbf R}^S_{i}\rightarrow 0$ as $S\rightarrow\infty$, there exists an integer $S_0>0$ such that
$\left<
{\mathbf R}^S_{i} \cdot \tilde{{\cal A}} {\mathbf R}^S_{i}\right> \leq \varrho/2$ holds for any $S>S_0$.
Thus, we have $\tilde {\cal E}^{n+1}\leq \tilde {\cal E}^n$, which completes the proof of energy dissipation law.
\end{proof}

\section{Newton--Krylov--Schwarz solver with adaptive time stepping}

By discretizing the Ni-based phase field system with the proposed energy stable scheme \eqref{NSOA-1}, 
a discrete nonlinear equation system ${\cal F}(\mathbf{X}^{n})=0$ is  
constructed and required to be solved at each time step. 
We omit the superscript $n$ in the 
remainder of the section. The unknown $\mathbf{X}$ is organized by points, i.e.  
$\mathbf{X}=(\tilde c_{1},\, \tilde \eta_{1}$, $ \tilde{u}_{1},$ $ \tilde c_{2},\, \tilde\eta_{2}$, $ \tilde {u}_2,\, \cdots)^{T}$, to maintain the strong coupling of the concentration field and the order parameter.
We solve the nonlinear system on a parallel supercomputer 
with the Netwon-Krylov-Schwarz (NKS) type algorithm \cite{NKS}.
Assume $\mathbf{X}_{m}$ is the current approximate solution, then a new solution $\mathbf{X}_{m+1}$ can be obtained by the following steps.
 \begin{itemize}
 \item Compute {an} inexact Newton direction $\mathbf{S}_m$ by solving the following Jacobian system:
 \begin{equation}\label{eq:line-system}
\begin{array}{ll}
J_m\mathbf{S}_m=-{\cal F}(\mathbf{X}_m),
\end{array}
\end{equation}
until 
$\| J_m \mathbf{S}_m+{\cal F}(\mathbf{X}_m)\| \leq  \textnormal{max} \{\xi_r\|{\cal F}(\mathbf{X}_m)\|, \xi_a\}$.
Here  $J_m:=\nabla {\cal F}(\mathbf{X}_m)$ is the Jacobian matrix at the solution $\mathbf{X}_m$, and 
 $\xi_r, \xi_\alpha \geq 0$ are the relative and absolute tolerances for the linear iteration, respectively. 

 \item Calculate a new solution through $\mathbf{X}_{m+1}=\mathbf{X}_{m} +\lambda_m\mathbf{S}_m$, where the step length $\lambda_{m}\in (0,1]$ is determined  by a line search procedure \cite{NMFU}.
\end{itemize}
The   stopping condition for the nonlinear iteration is set as
\begin{equation}
\|{\cal F}(\mathbf{X}_{m+1})\| \leq \max\{\varepsilon_r
\|{\cal F}(\mathbf{X}_0)\|, \varepsilon_a\},
\end{equation}
where $\varepsilon_r, \varepsilon_a \geq 0$ are the relative and absolute tolerances for the 
nonlinear iteration, respectively. 

In the NKS algorithm, the Jacobian system \eqref{eq:line-system} is solved by using an  iterative method based on Krylov subspace together with an additive Schwarz type preconditioner. 
In particular, we employ the restarted Generalized Minimal Residual (GMRES) method \cite{gmres} as the iterative solver and consider three preconditioners including
the classical additive Schwarz (AS) preconditioner 
\cite{DDA},  the left restricted additive Schwarz (left-RAS, \cite{cai99sisc}) 
preconditioner, and the right restricted  additive Schwarz (right-RAS, \cite{cai03sinum_ash}) 
preconditioner, for comparison purpose. We refer the readers to \cite{cai03sinum_ash, DDA, cai99sisc}  for more details about the three preconditioners.

The Ni-based alloys phase field system usually contains 
multiple time scales that may vary in orders of magnitude during the nucleation, the growth, and the coarsening of $\gamma'$ particles. 
Therefore an adaptive control of the time step size is necessary, 
in which the time step size is selected based on the desired solution accuracy and the dynamic 
features of the system. According to \cref{theorem2}, the implicit scheme \eqref{NSOA-1}
is unconditional energy stable and the time step size can be chosen from a wide range.
Following the idea of  \cite{huang2020parallel, qiao2011adaptive, PFCDVD, zhang2013},  the time step size at  the $n $-th time step is predicted to
\begin{equation}
\label{sencondt}
\tilde{\Delta} t_n= \max \left(\Delta t_{\min},\frac{\Delta t_{\max}}{\sqrt{1+\zeta \mathbf{X}_d'(t_n)^2}} \right),
\end{equation}
where $\mathbf{X}_d'(t_n) =\frac{\|\mathbf{X}^{n}-\mathbf{X}^{n-1}\|}{\Delta t_{n-1}}$ 
corresponds to the change rate of numerical solutions on the two previous time steps, 
and $\zeta$ is a positive pre-chosen parameter. In \eqref{sencondt}, $\Delta t_{\min}$ and 
$\Delta t_{\max}$ are defined as the lower and upper bounds of the time step size.
Then we use the NKS algorithm  to solve the discrete system \eqref{NSOA-1} with the predicted time step size $\tilde{\Delta} t_n$. If the NKS solver diverges 
with the currently predicted time step size $\tilde{\Delta} t_n$, a smaller predicted time step size $\tilde{\Delta} t_n:=\tilde{\Delta} t_n/\sqrt 2 $ is chosen to restart the NKS algorithm and $\zeta:=2\zeta$.
The loop is broken down and the time step size $\Delta t_n$ is set to be $\tilde{\Delta} t_n$ until the NKS solver converges.

\section{Numerical experiments}
In this section, we investigate the numerical behavior and parallel performance of the proposed algorithm for the Ni-based alloys phase field system \cref{system-1}.
We carry out several three dimensional tests to validate the discretization of the proposed algorithm and study the morphology evolution of $\gamma'$ precipitates in Ni-based alloys during aging process. 
We mainly focus on (1) the morphology evolution of $\gamma'$ particles during the nucleation, the growth, and the coarsening; (2) the performance of the NKS algorithm
with different preconditioners; and (3) the parallel scalability of the proposed algorithm.

The algorithm for the Ni-based alloys phase field system is implemented on top of the Portable, Extensible Toolkits for Scientific computations (PETSc) library \cite{balay2019petsc}.
The numerical experiments are carried out on the BSCC-A6 supercomputer.
The computing nodes of BSCC-A6 are
comprised of two 32-core AMD CPUs with 256 GB local memory,
and are interconnected via a proprietary high performance network. In the numerical
experiments we use all 64 CPU cores in each node and assign one subdomain to each
core.
 In the discrete scheme \cref{NSOA-1}, we set
$S = 10$ so that the error coming from the Taylor approximation can be ignored.
The stopping conditions for the nonlinear and linear iterations are set as 
 $\varepsilon_r =  10^{2} \xi_r= 1\times10^{-8}$ and
 $\varepsilon_a =  10^{-3} \xi_a= 1\times10^{-6}$, respectively.

\subsection{The shape evolution of a single $\gamma'$ particle}
The shape of the $\gamma'$ precipitate mainly depends on the balance of the interfacial energy and the elastic energy.
To understand the elastic energy contribution to the shape evolution of the $\gamma'$
precipitate, we run a set of simulations by multiplying Hooke's tensor by a dimensionless parameter ${\cal L}$,
which changes the ratio of the  elastic energy and the interfacial energy.  According to \cite{mueller19983d}, 
doubling the elastic energy $E_{\textnormal{el}}$ achieves a similar effect of the double average characteristic length of the particles.
In all simulations, the computational domain $\Omega$ is covered by an $80\times80\times 80$ uniform mesh with $h=\Delta x=\Delta y =\Delta z=0.25$.
The time step size is initially set as $\Delta t_1=0.01$ and then adaptively controlled with 
$\zeta =100$, $\Delta t_{\textnormal{min}}=0.01$, and $\Delta t_{\textnormal{max}}=2$.
A spherical $\gamma'$ particle is embedded in the $\gamma$ matrix and located at the center of computational domain $\Omega$ with radius $r=7.5$. 
The initial values are set as $c=0.238,\ \eta =0.01$  for  $\gamma'$ phase  and $c=0.1375,\ \eta =0.99$ for  $\gamma$ phase, respectively.

Four simulations with ${\cal L}=0,\,1,\,3,$ and 5 are considered. The simulation with ${\cal L}=0$ is corresponding to a standard phase field model without the elastic energy, which is
taken as a comparison. 
The simulations with ${\cal L}=1,\,3,\,5$ are  equivalent to the simulations with initial radius $r=7.5,\,22.5,$ and $37.5$, respectively.
The shape evolutions of a single $\gamma'$ particle for the four simulations are given in Fig.~\ref{fig:example1}. As the particle size increases, the equilibrium  shape of the $\gamma'$ particle
transitions from spherical to approximately cubic.
We plot the evolutions of the total free-energy, the interfacial energy, and the elastic energy for ${\cal L}=1,\,3,\,5$ in Fig.~\ref{fig:example1.2}, which clearly indicates that the proposed scheme satisfies the energy dissipation law even with a relatively large time step size.
By combining Fig.~\ref{fig:example1} and Fig.~\ref{fig:example1.2}, we found that the interfacial energy will dominate in the decay of the total energy, and the equilibrium particle shape
will be near-spherical as the particle size is small. On the contrary, the elastic energy gradually dominates in the decay of the total energy, and the equilibrium particle shape
will approach a cube as the particle size becomes larger.
The history of the time step size is also displayed in \ref{fig:example1.2},
which shows the time step size is successfully adjusted from 
$\Delta t_{\min}$ to $\Delta t_{\max}$ by two orders of magnitude. 

\begin{figure}[H]
\centering
{\subfloat[$t=10$]{\includegraphics[width=0.37\textwidth]{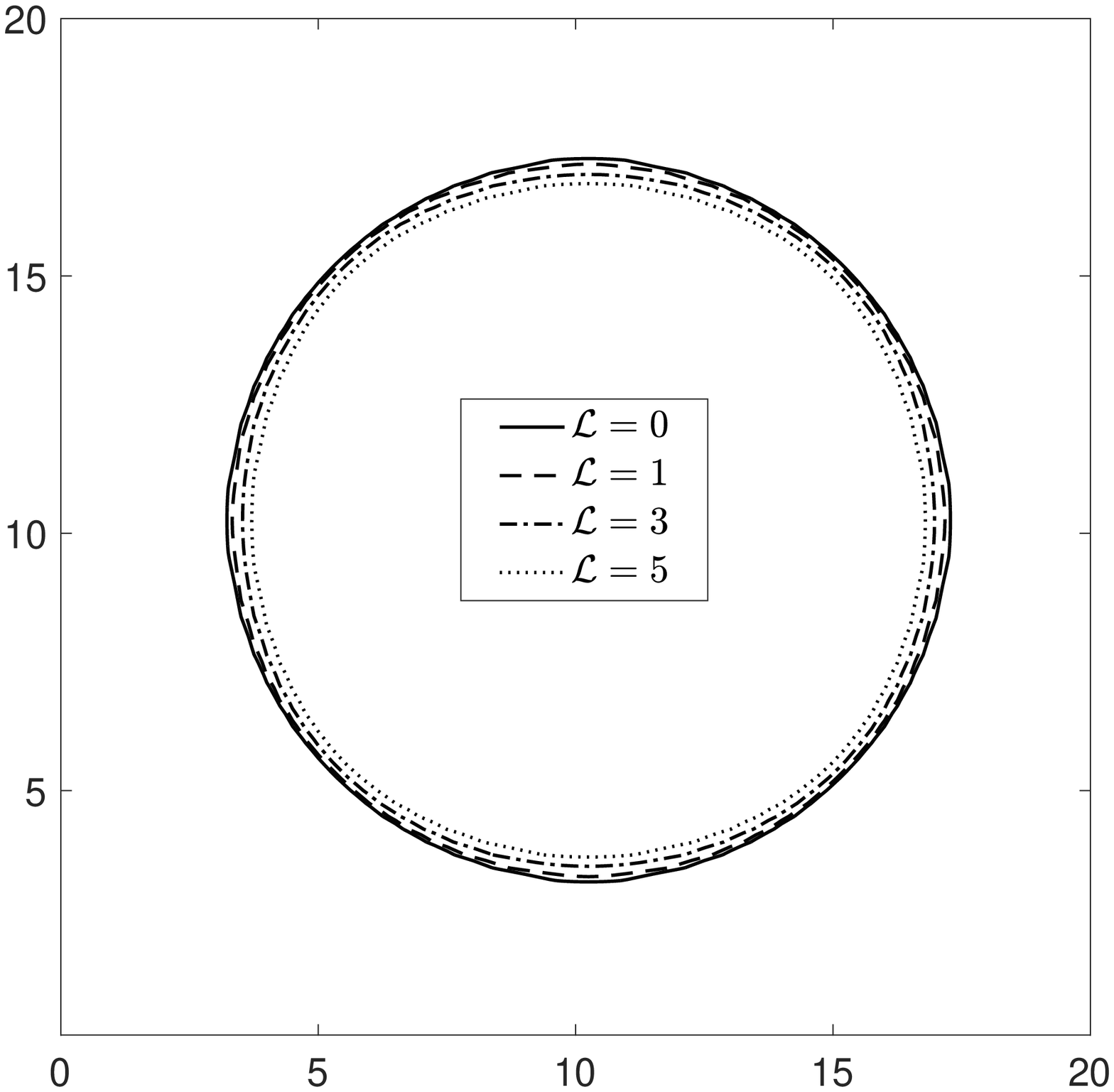}}}~
\subfloat[$t=50$]{\includegraphics[width=0.37\textwidth]{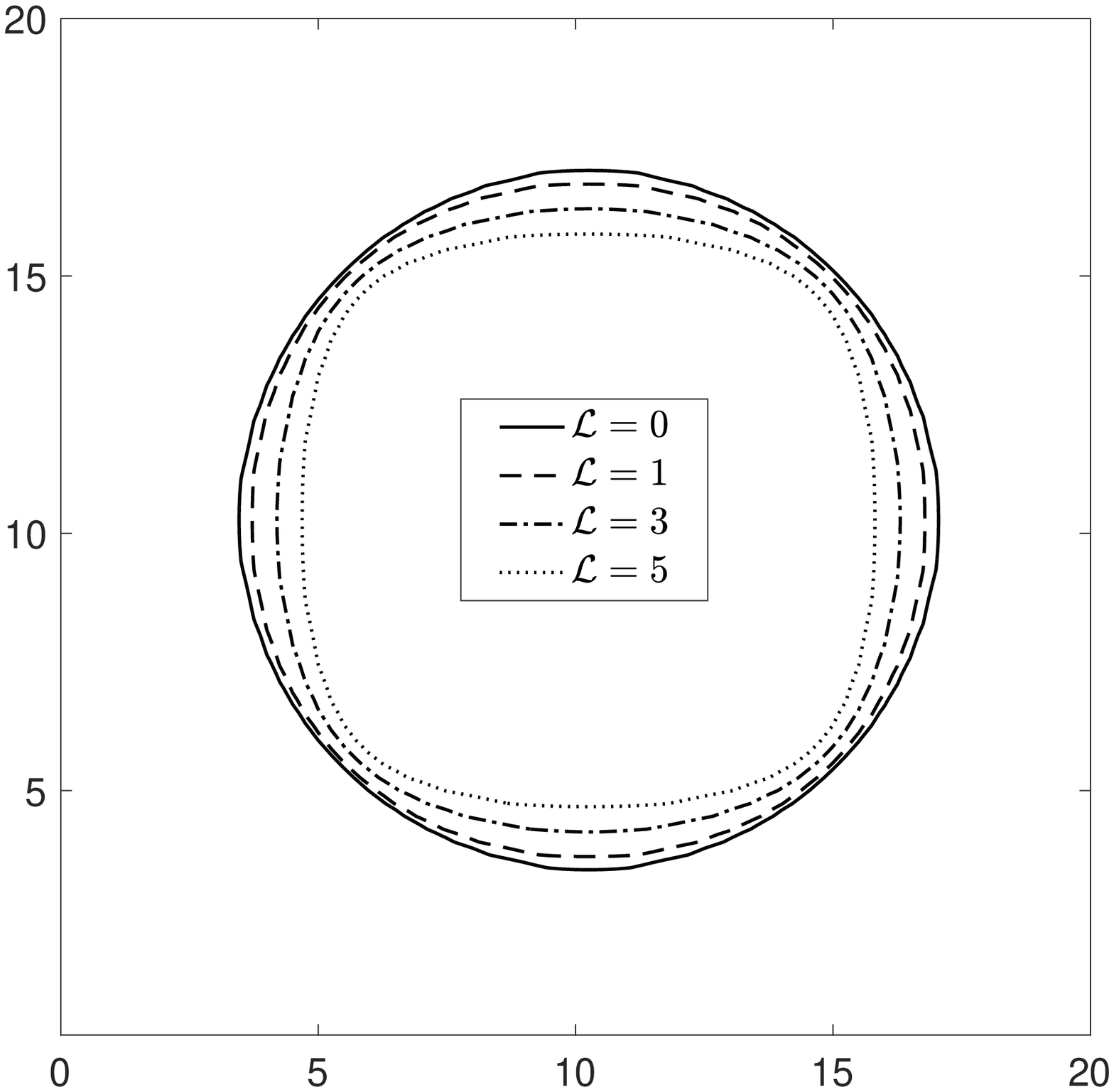}}~\\
\subfloat[$t=100$]{\includegraphics[width=0.37\textwidth]{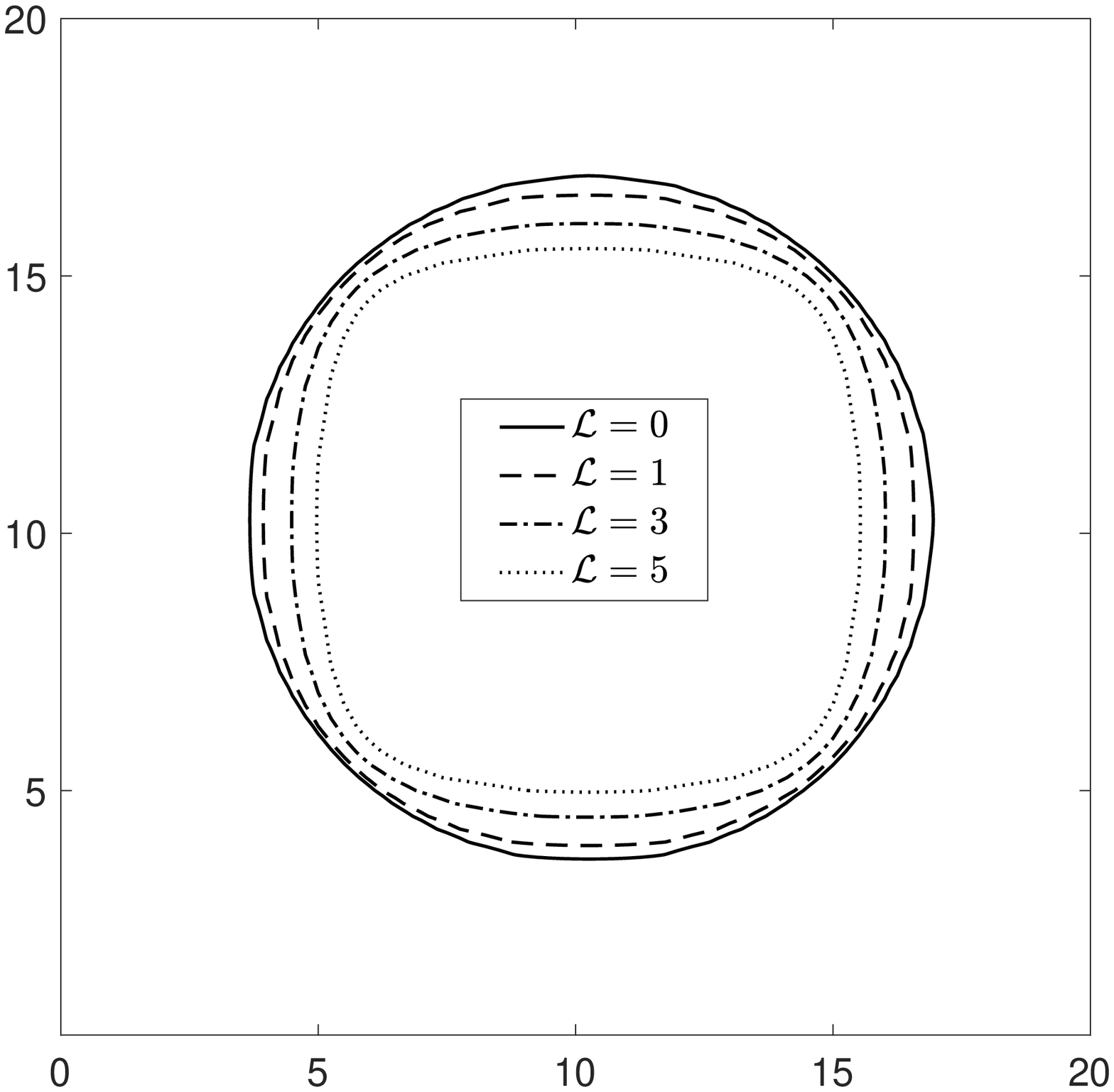}}~
\subfloat[$t=300$]{\includegraphics[width=0.37\textwidth]{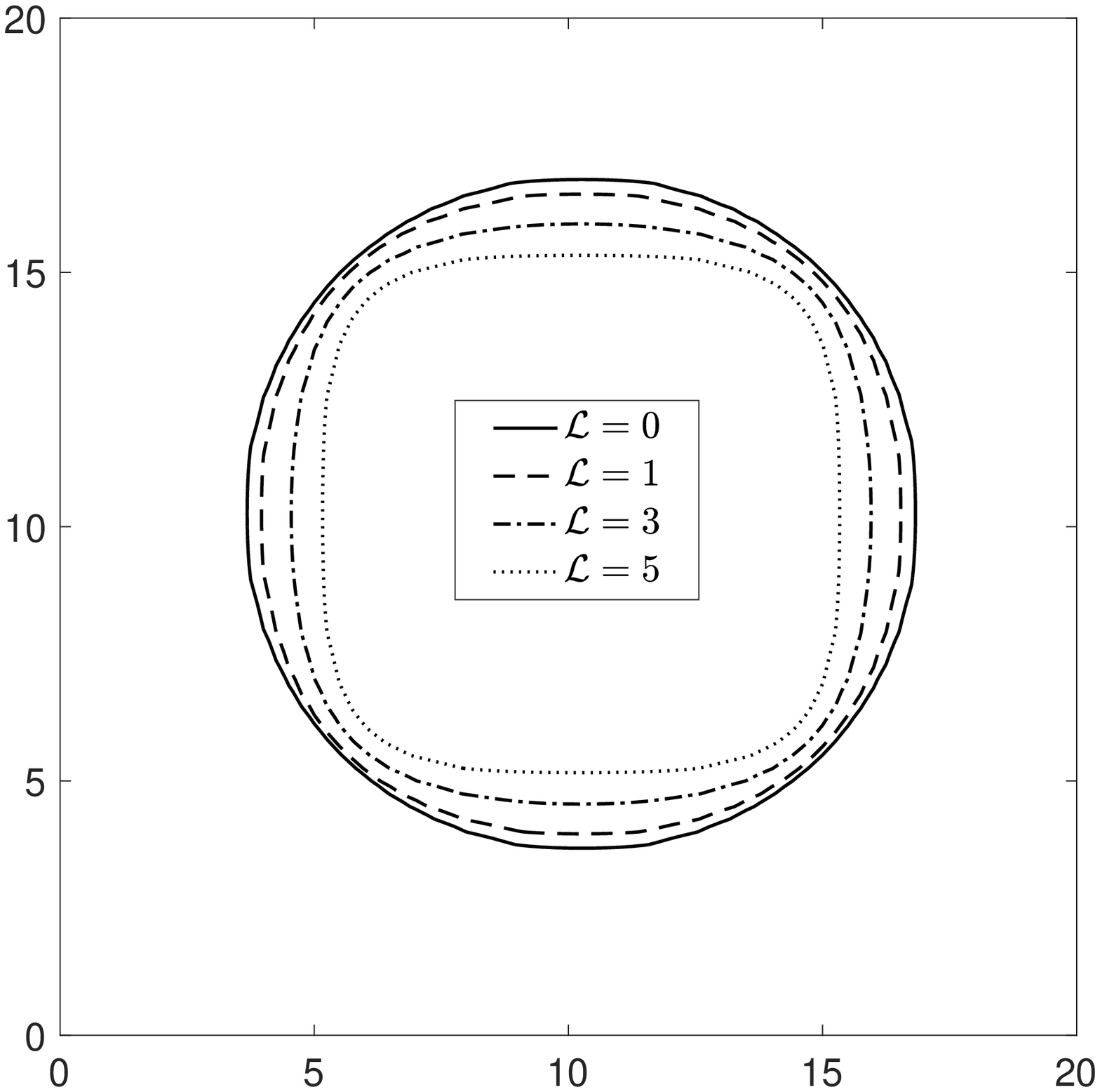}}
\caption{The shape evolutions of a single $\gamma'$ particle with ${\cal L}=0,\,1,\,3,$ and 5. The contour plots for $c=0.22$ at $x=10$ are displayed.}\label{fig:example1}
\end{figure}

\begin{figure}[H]
\centering
\subfloat[${\cal L}=1$]{\includegraphics[width=0.325\textwidth]{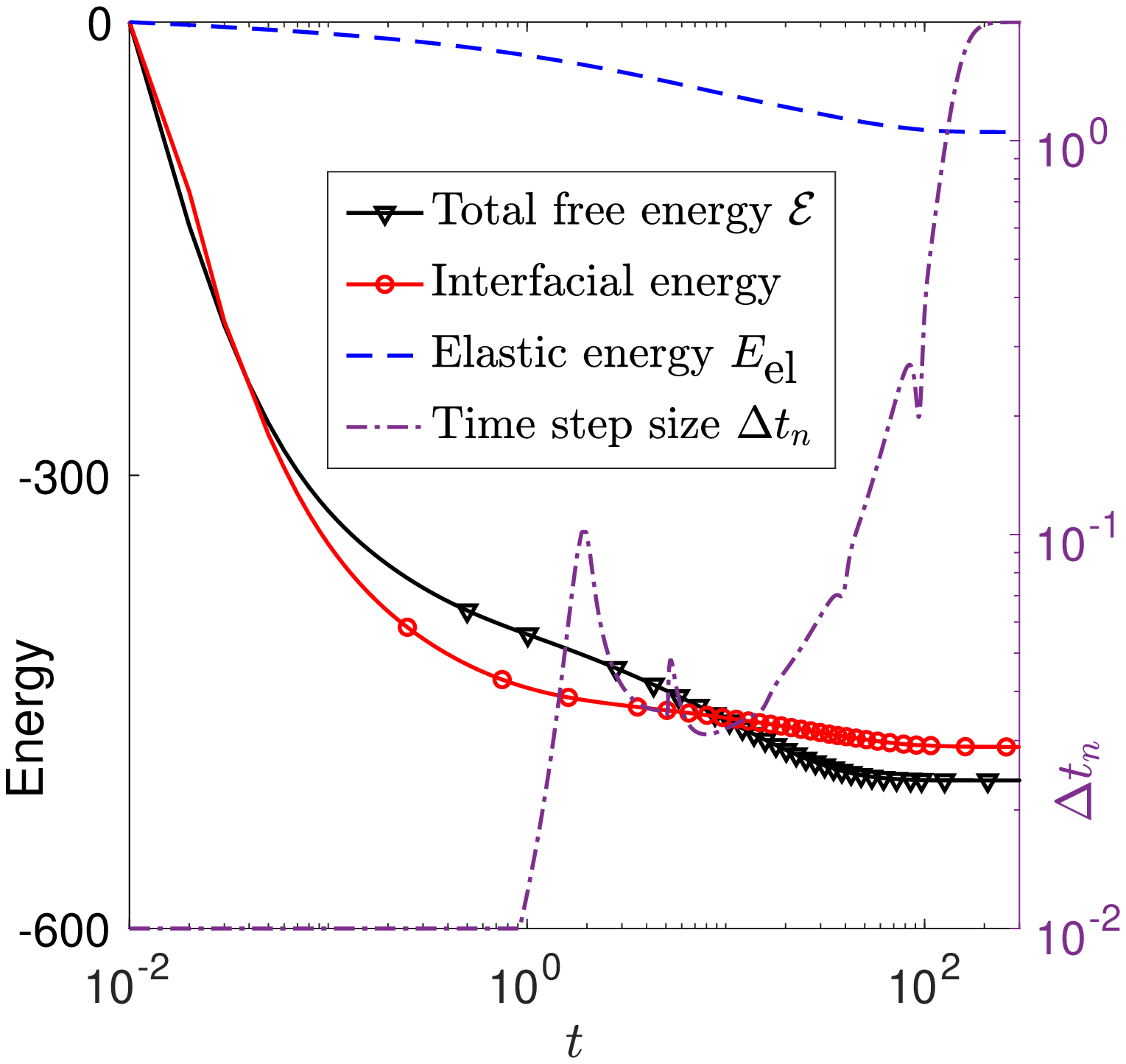}}~
\subfloat[${\cal L}=3$]{\includegraphics[width=0.325\textwidth]{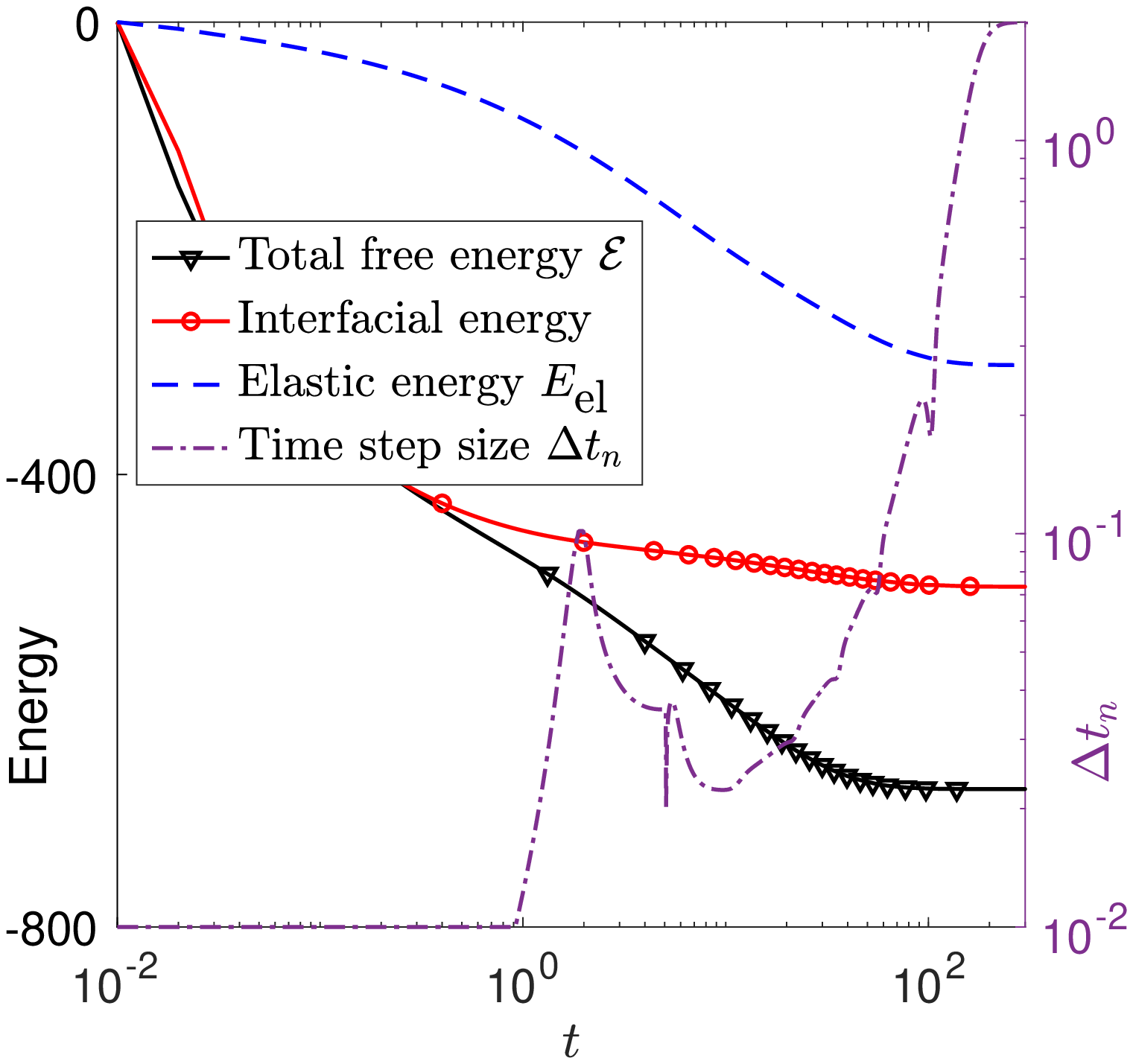}}~
\subfloat[${\cal L}=5$]{\includegraphics[width=0.325\textwidth]{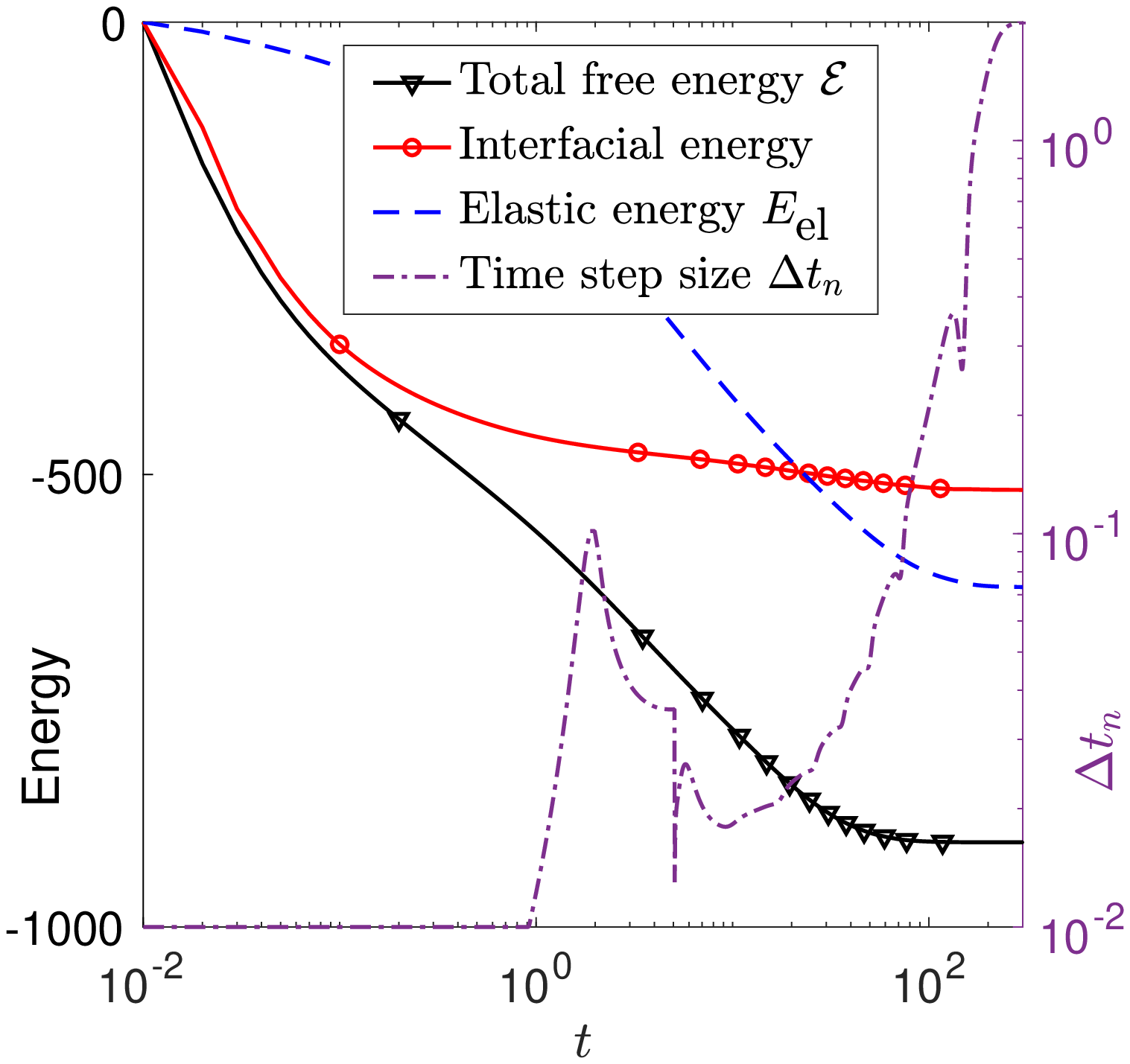}}
\caption{The evolutions of the relative total free energy $\tilde{\cal E}^n-\tilde{\cal E}^1$, the relative interfacial energy, the relative  elastic energy 
(left axis), and the history of the time step size  (right axis).}\label{fig:example1.2}
\end{figure}

\subsection{Coarsening rate constant of a single $\gamma'$ particle}
The precipitation process of the $\gamma'$ phase can be roughly divided into three stages: the nucleation, the growth, and the coarsening.
In this subsection, we focus on the coarsening rate constants of a single $\gamma'$ particle.
According to Lifshitz--Slyozov--Wagner theory \cite{lifshitz1961kinetics}, the increase in the
mean particle size as a function of time will follow a cubic
growth law 
\begin{equation}\label{growth law}
\left<R\right>^3-\left<R_0\right>^3=K(t-t_0),
\end{equation}
where $\left<R\right>$ is the mean particle size, $\left<R_0\right>$ and $t_0$ are the mean particle size and time at the beginning of the coarsening. 
$K$ is known as the coarsening rate constant, which depends on the surface free energy of
the particle-matrix interface and the diffusion coefficient of the solute in the matrix. 
To verify the growth law \eqref{growth law}, we run simulations with different volume fractions of $\gamma'$ phase,
which are $V_f=10\%,$ $17.5\%,$  $25\%,$  $32.5\%,$  and $40\%,$ respectively.
The computational domain is covered by a $60\times 60\times 60$ uniform mesh size. The mesh size and parameters used in the adaptive time stepping strategy are the same as in the previous simulations, except that the initial time step size is set to $\Delta t_{\textnormal{min}}=0.001$.  
Initially, a spherical $\gamma'$ particle with $c=0.234$ is embedded in a matrix with $c=0.147+0.087 V_f$. The particle is located at the center of $\Omega$ with radius $r=1.5$. 

    The relationship between  $R^3$ and $t$ is illustrated in Fig.~\ref{fig:example1.3}-(b).
     Through  data fitting, the dimensionless coarsening rate constants of a single $\gamma'$ particle are 0.24, 0.83, 1.75, 3.45, and 5.25, which are corresponding to $V_f=10\%,$ $17.5\%,$  $25\%,$  $32.5\%,$  and $40\%,$ respectively. 
     It is obvious that the dimensionless coarsening rate constant increases as the volume fraction of $\gamma'$ increasing. 
     The evolutions of the total free energy, the interfacial energy, the elastic energy, and the average composition for the case with $V_f=0.4$ are shown in Fig.~\ref{fig:example1.3}-(a). 
     The results in other cases are very similar and will not be repeated here.
     From Fig.~\ref{fig:example1.3}-(a), we found that the  interfacial energy and the elastic energy increase as the particle gets larger, but the total free keeps decaying due to the energy dissipation law and unconditional energy stability of the newly presented scheme. 
     We also plot the average composition $\bar c$ in Fig.~\ref{fig:example1.3}-(a), which  verifies that  the newly proposed scheme fully complies with the  mass conservation  law.

\begin{figure}[H]
\centering
\subfloat[$V_f=0.4$]{\includegraphics[width=0.45\textwidth]{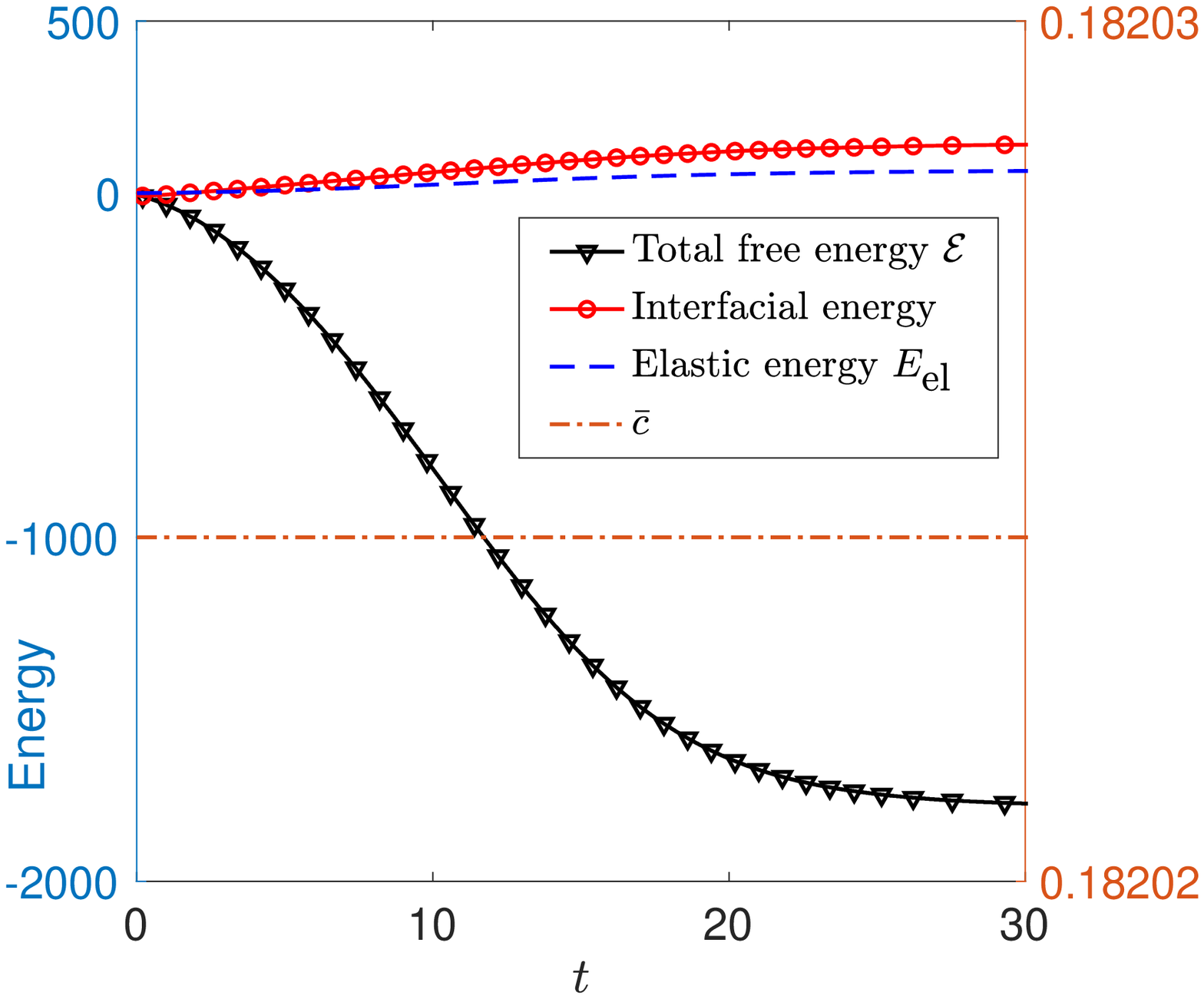}}~
\subfloat[]{\includegraphics[width=0.45\textwidth]{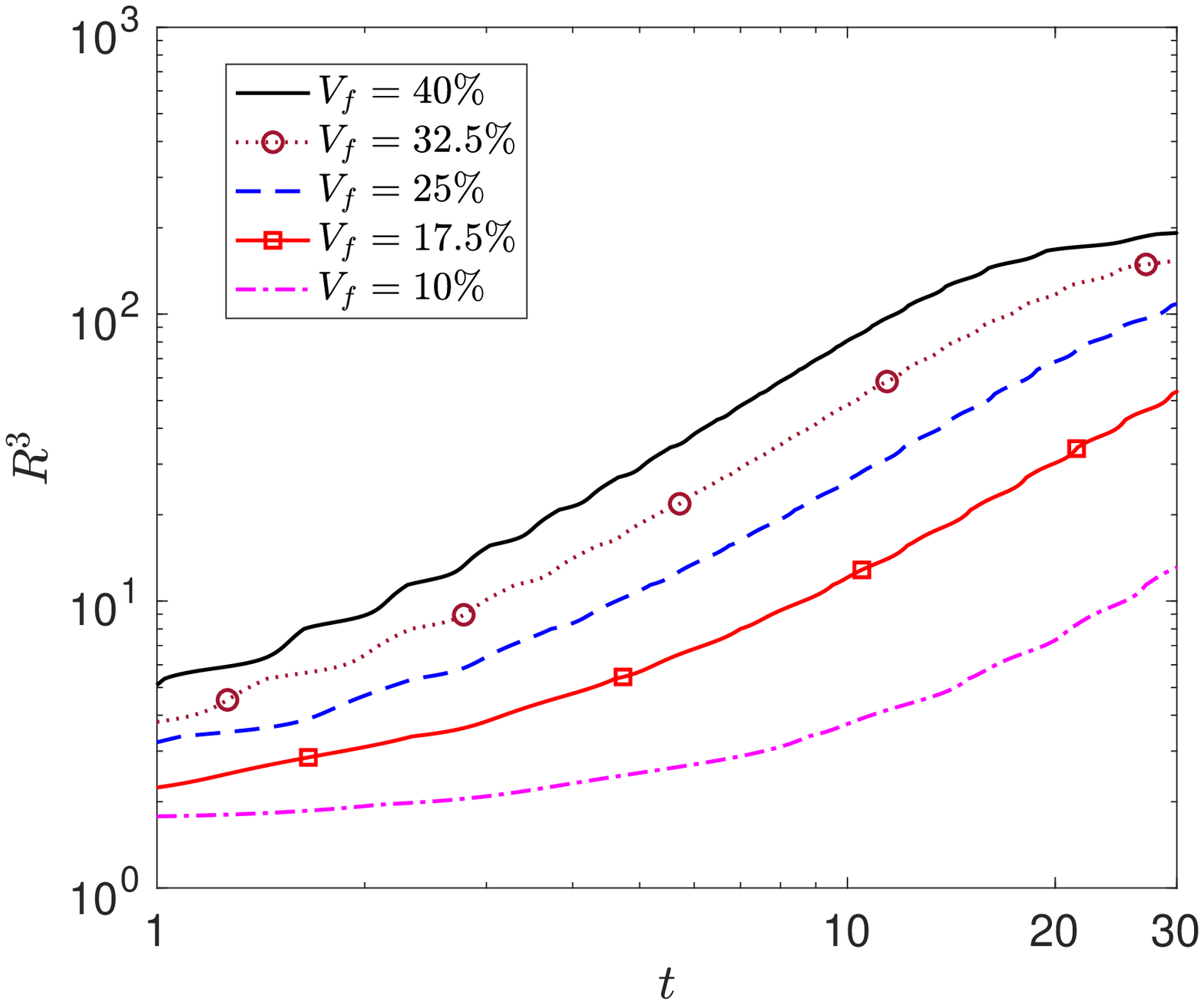}}
\caption{(a) The evolutions of the total free energy, the interfacial energy,  the elastic energy (left axis), and the average composition $\bar c$  (right axis);
(b) Evolutions of  $R^3$ for  different volume fractions.}\label{fig:example1.3}
\end{figure}

\subsection{Morphological evolution of the $\gamma'$ phase and particles number density change with increased volume fraction}
In this subsection, we investigate the morphological evolution and particle number density changes of the $\gamma'$ phase by simulating several tests for different volume fractions. 
The computational domain $\Omega=[0,\,75]\times[0,\,75]\times[0,\times10]$ is covered by a $300\times 300\times 40$ uniform mesh with $h=0.25$.
The time step size is initially set as $\Delta t_1=0.01$ and then adaptively controlled with 
$\zeta =100$, $\Delta t_{\textnormal{min}}=0.01$, and $\Delta t_{\textnormal{max}}=2$.
A homogeneous solution with small composition fluctuations around the average composition is taken as the initial value, which is set to be a randomly distributed state $(c^0,\, \eta^0)=(\bar{c}+\delta_c, 0.1+\delta_\eta)$,  where $\delta_c$ or $\delta_\eta $ is a uniform random distribution function of $-0.05$ to 0.05.
 The value of the average composition $\bar {c}$ 
is determined by the volume fraction of the $\gamma'$ phase. For the reason of comparison, we run four simulations with $\bar{c}\approx  0.1622$, 0.1687, 0.1753, and 0.1818, which are corresponding to the volume fraction $V_f=17.5\%$, $25\%$, $32.5\%$, and $40\%$, respectively. All simulations are stopped at $t=200$. 

The distributions of the $\gamma'$ phase at $t=200$ for the four simulations are given in Fig.~\ref{fig:example1.4}, from which we found a remarkable morphology change of $\gamma'$ phase with increased volume fraction. The particles size, shape, and number density change as the volume fraction of the $\gamma'$ phase increases.
Next, we take the simulation with $V_f=17.5\%$ as an example to study the morphological evolution and particle number density of the $\gamma'$ phase.
The pseudo-color plots of the composition $c$ at $t=12,\,39$, and $171$ are displayed in Fig.~\ref{fig:example1.5} (a)-(c), which are corresponding to three stages: the nucleation, the growth, and the coarsening, respectively.
The nucleation of new $\gamma'$ phase forms randomly due to small composition fluctuations.  The growth of these nuclei is caused by diffusive  transportation, and the coarsening involves the growth of large particles at the expense of the dissolution of small ones, driven by an overall reduction in the interfacial energy.
Contour plots of the third component of displacement, $u_3$,  on surface $z=0$ are also shown in  Fig.~\ref{fig:example1.5} (d)-(f), from which we conclude that the elastic interaction is long range and can lead to the spatial correlation between the $\gamma'$ particles.

\begin{figure}
\centering
{\includegraphics[width=0.07\textwidth]{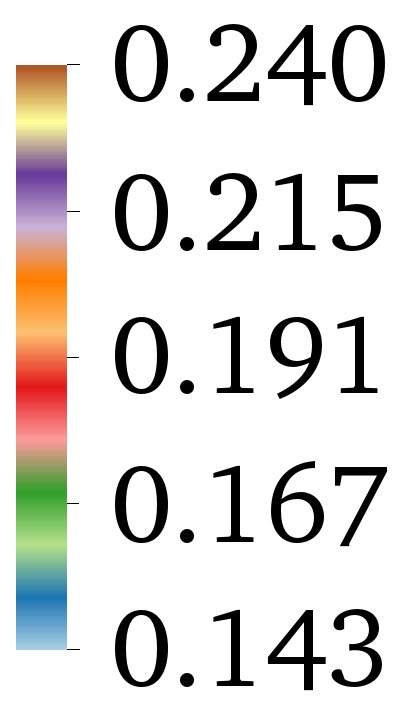}}~
\subfloat[$V_f=17.5\%$]{\includegraphics[width=0.22\textwidth]{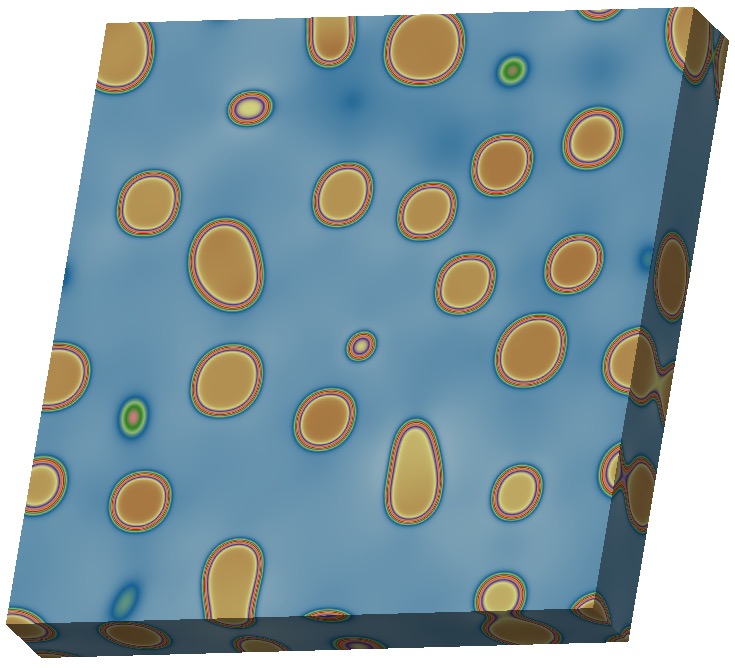}}~
\subfloat[$V_f=25\%$]{\includegraphics[width=0.22\textwidth]{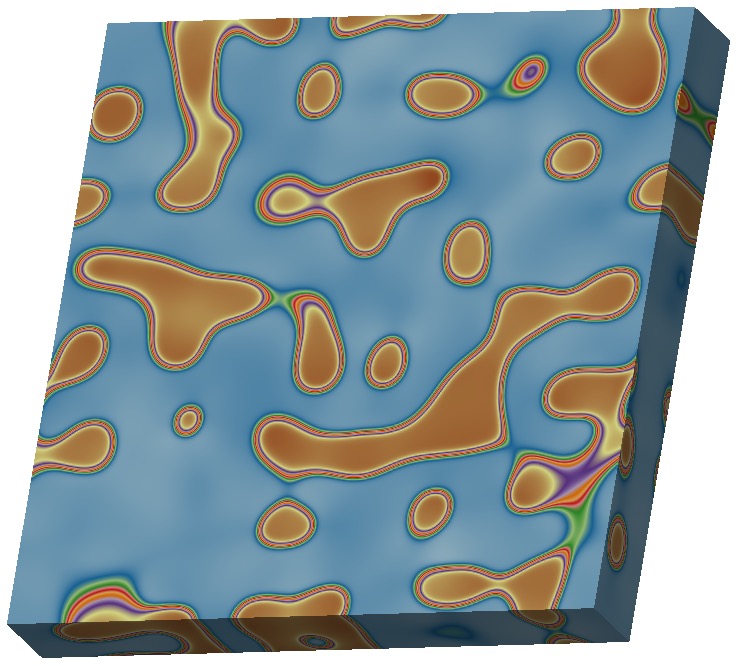}}~
\subfloat[$V_f=32.5\%$]{\includegraphics[width=0.22\textwidth]{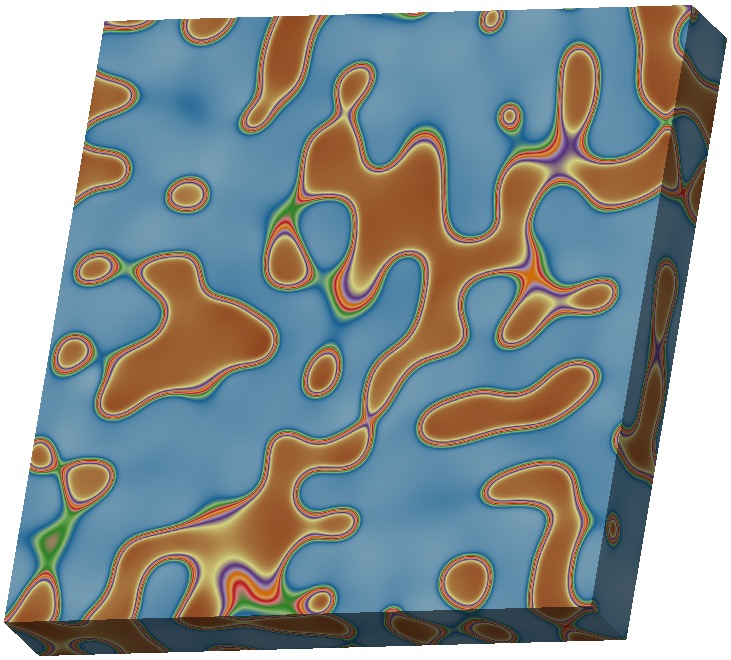}}~
\subfloat[$V_f=40\%$]{\includegraphics[width=0.22\textwidth]{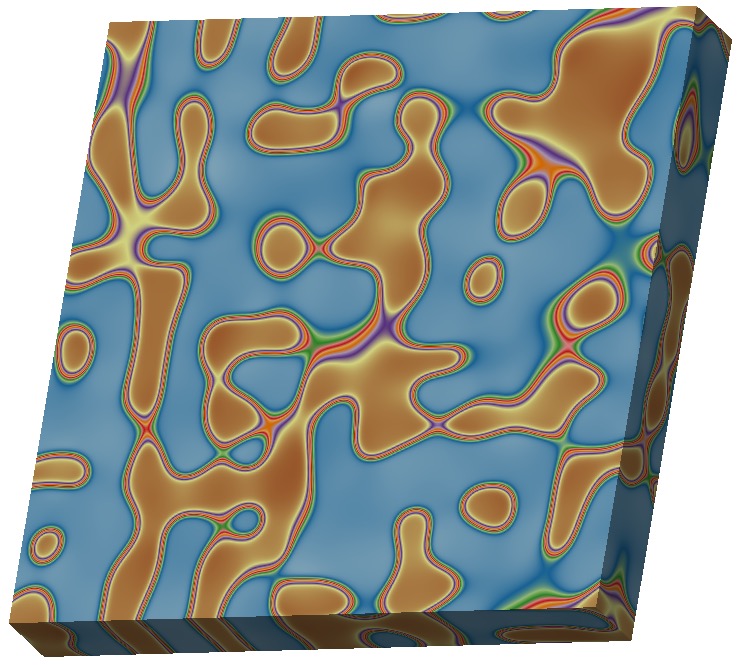}}~\\
{\includegraphics[width=0.07\textwidth]{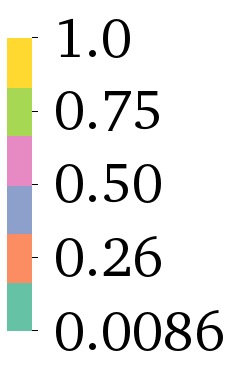}}~
\subfloat[$V_f=17.5\%$]{\includegraphics[width=0.22\textwidth]{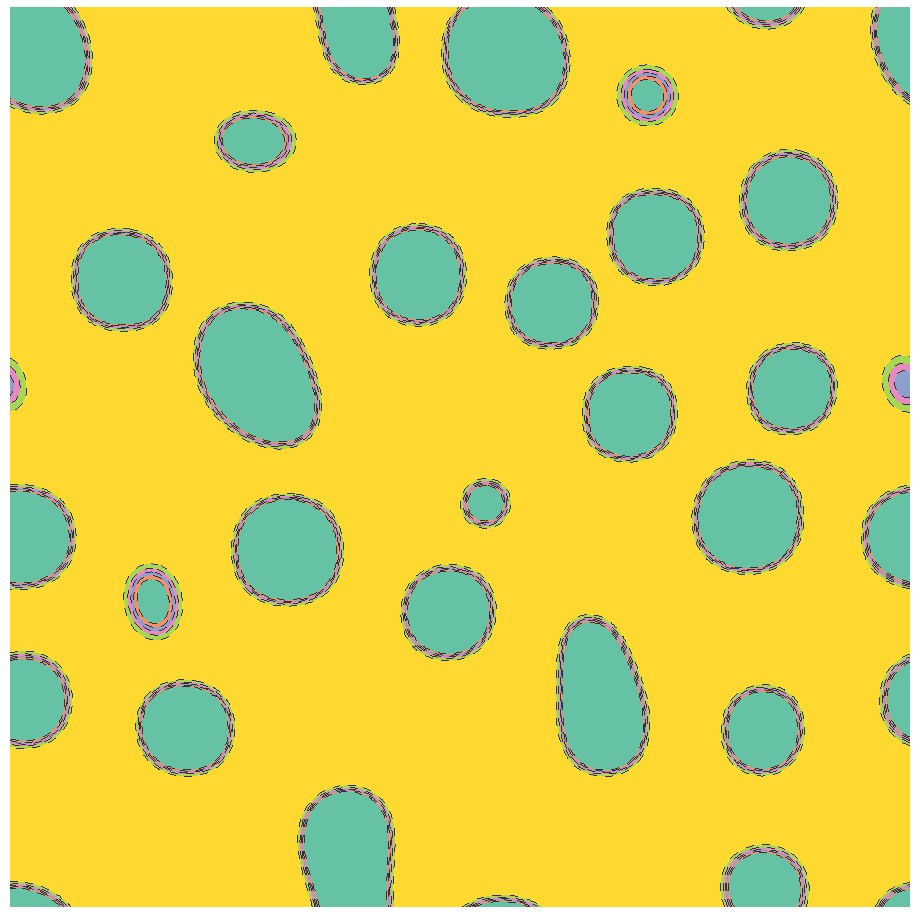}}~
\subfloat[$V_f=25\%$]{\includegraphics[width=0.22\textwidth]{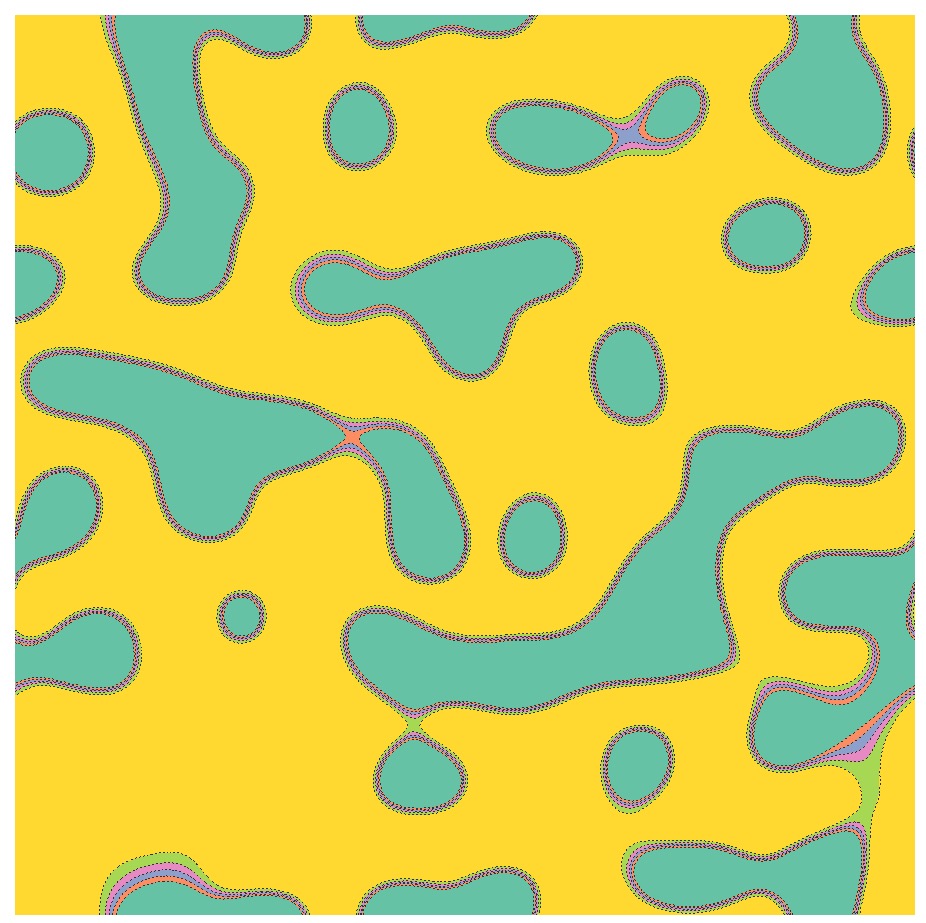}}~
\subfloat[$V_f=32.5\%$]{\includegraphics[width=0.22\textwidth]{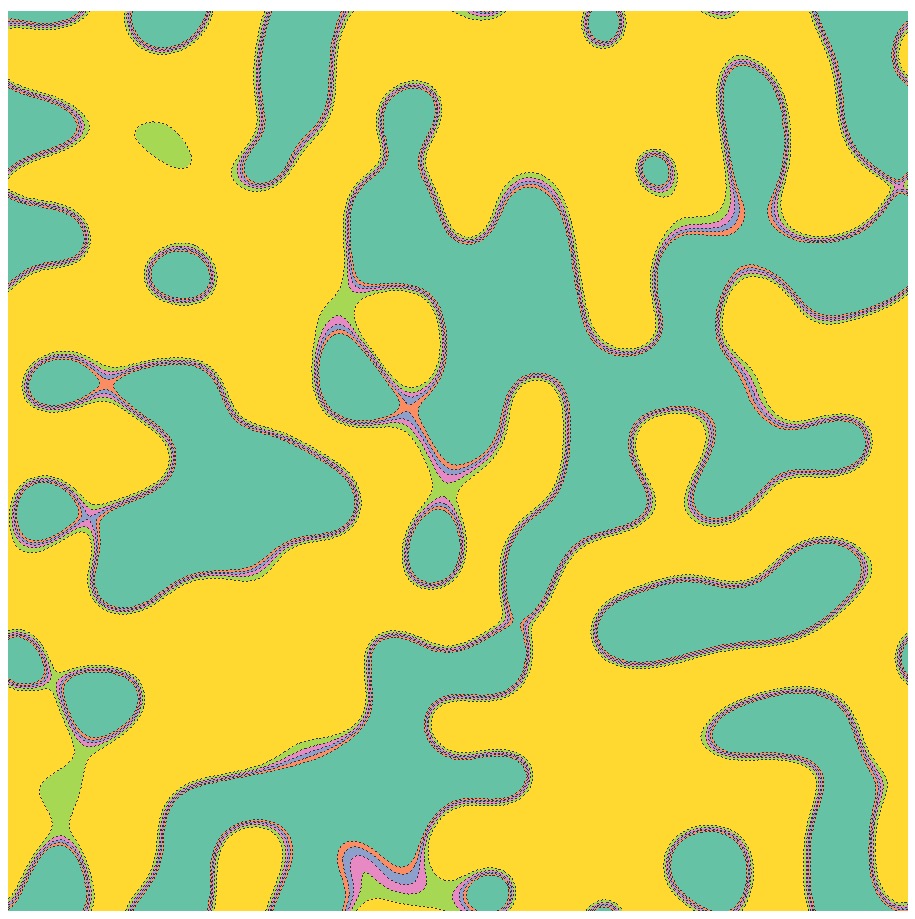}}~
\subfloat[$V_f=40\%$]{\includegraphics[width=0.22\textwidth]{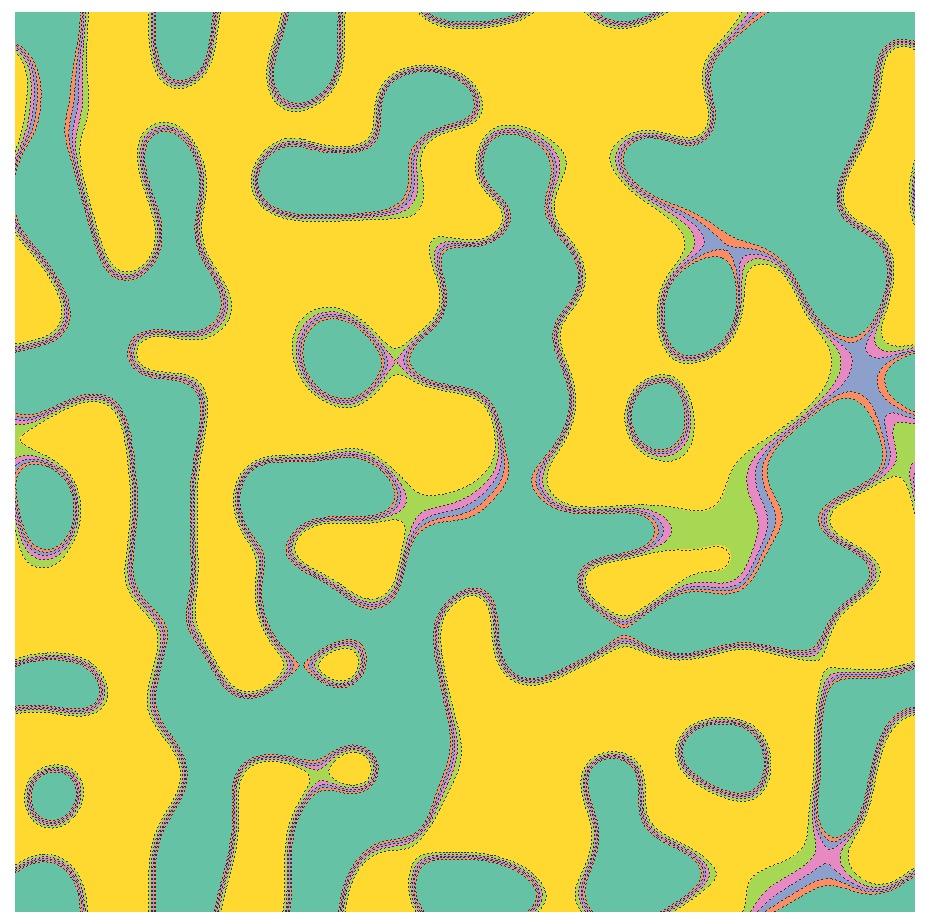}}~
\caption{(a)-(d): Pseudo-color plots of composition $c$  at $t=200$ with different  volume fractions of the $\gamma'$ phase.
(e)-(h): Contour plots of long range order parameter $\eta$ on surface  $z=0$ at $t=200$.}\label{fig:example1.4}
\end{figure}

\begin{figure}
\centering
{\includegraphics[width=0.07\textwidth]{ccolorbar.jpg}}~
\subfloat[$t=12$, nucleation]{\includegraphics[width=0.3\textwidth]{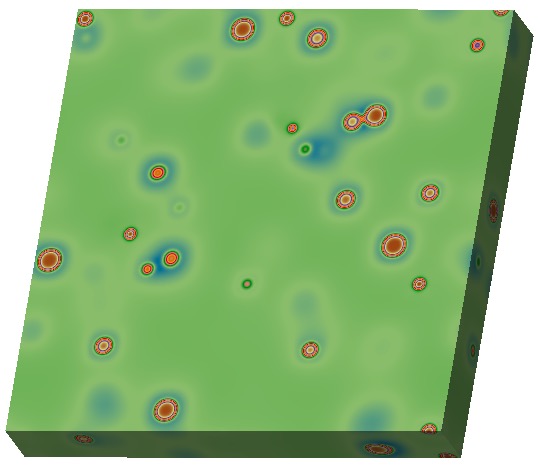}}~
\subfloat[$t=39$, growth]{\includegraphics[width=0.3\textwidth]{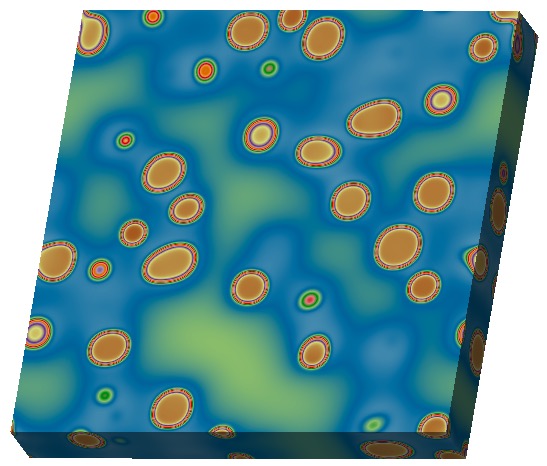}}~
\subfloat[$t=171$, coarsening]{\includegraphics[width=0.3\textwidth]{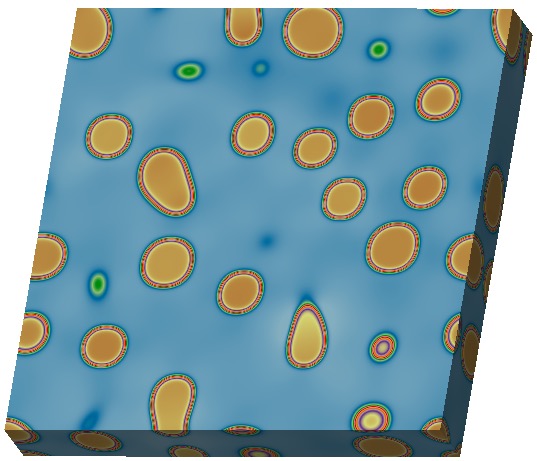}}~\\
{\includegraphics[width=0.07\textwidth]{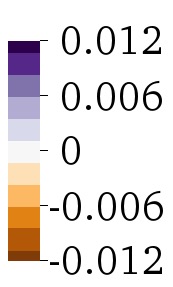}}~
\subfloat[$t=12$]{\includegraphics[width=0.3\textwidth]{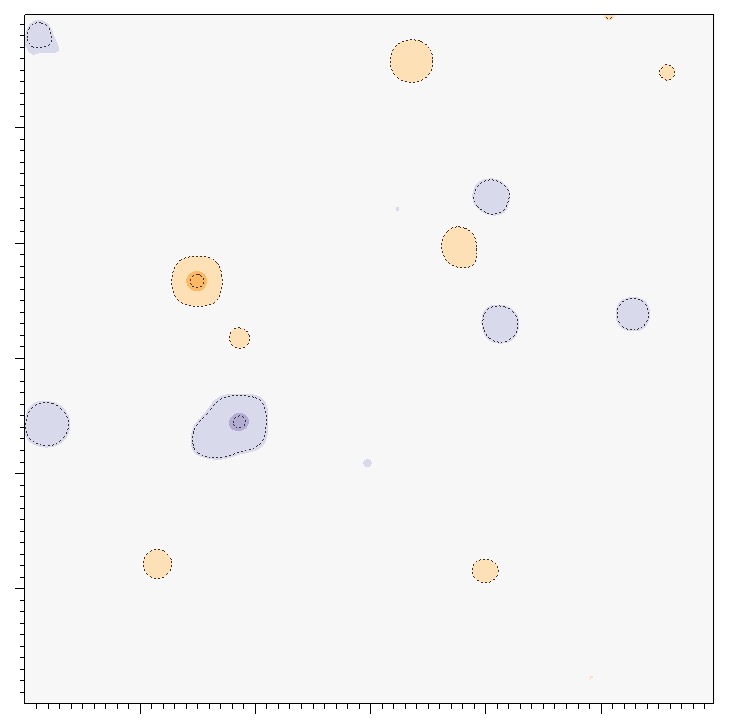}}~
\subfloat[$t=39$]{\includegraphics[width=0.3\textwidth]{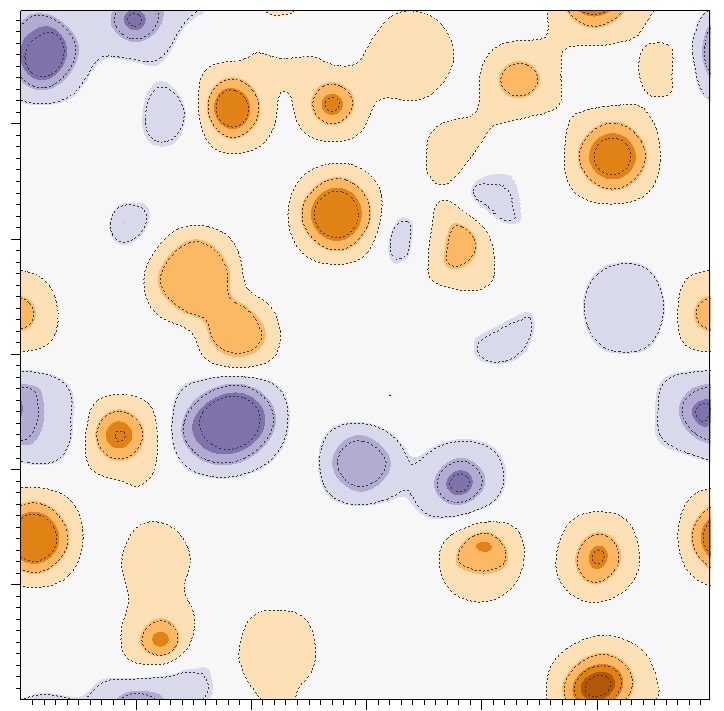}}~
\subfloat[$t=171$]{\includegraphics[width=0.3\textwidth]{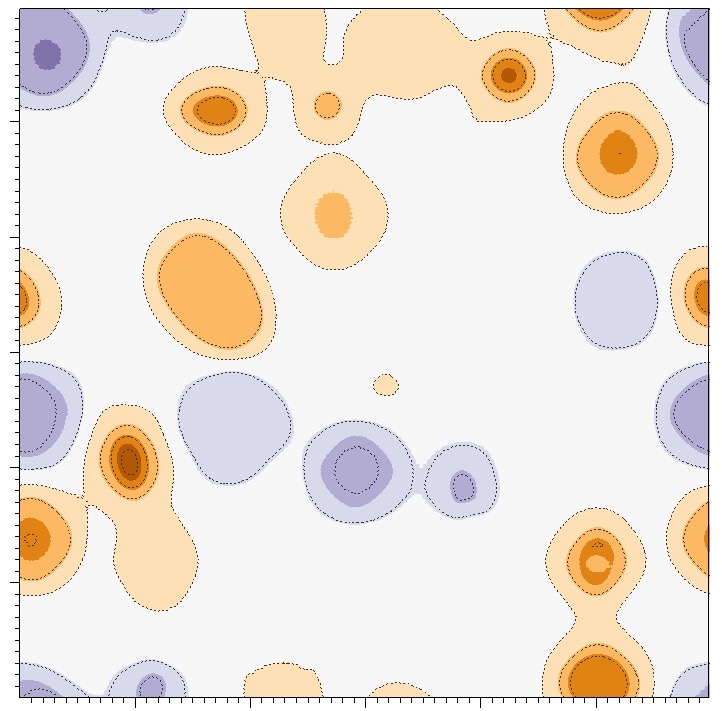}}~
\caption{Volume fraction $V_f=17.5\%$. (a)-(c): Pseudo-color plots of composition $c$  at  different  times.
(d)-(f): Contour plots of the third component of displacement $u_3$   on surface $z=0$ at  different  times.}\label{fig:example1.5}
\end{figure}

The evolutions of the total free energy, the interfacial energy, and the elastic energy  for the simulation with $V_f=17.5\%$ are shown in Fig.~\ref{fig:example1.6}-(a).
During the nucleation stage, the total free energy decreases rapidly as  the interfacial energy and the elastic energy increase. 
At the growth stage, the  total  energy decay rate becomes smaller as the interfacial energy and the elastic energy increase rapidly.
During the coarsening stage, the  total  energy decays continuously with the shape evolution of $\gamma'$ particles, which is mainly determined by the balance between the interfacial energy and the elastic energy.
We also plot the average composition $\bar c$ in Fig.~\ref{fig:example1.6}-(a), which clearly verifies that  the newly proposed scheme fully complies with the  mass conservation  law.
The evolution of  $\left<R^3\right>$ is given in Fig.~\ref{fig:example1.6}-(b), which verifies  the cubic growth law \eqref{growth law}. 
During the growth stage and early  coarsening stage, the slope $K$ in the cubic growth law \eqref{growth law} is near to 0.83, which is the coarsening rate constant of a single $\gamma'$ particle for volume fraction $V_f=17.5\%$. 
As coarsening proceeds, the coarsening rate constant decreases and eventually approaches zero when a steady state is reached. 
The number of $\gamma'$ particles as a function of time $t$ is also plotted in Fig.~\ref{fig:example1.6}-(b), from which we found that the particle number decreases rapidly during the growth stage. 
The temporal evolution of the normalized particle size distribution (PSD) of $R/\left<R\right>$ for the $\gamma'$ phase are shown in Fig.~\ref{fig:example1.7}, which shows a good agreement with the Lifshitz--Slyozov--Wagner  (LSW) theory \cite{lifshitz1961kinetics, wagner1961theory} and the experiment results reported in \cite{ardell1966coarsening}. 

\begin{figure}
\centering
\subfloat[]{\includegraphics[width=0.45\textwidth]{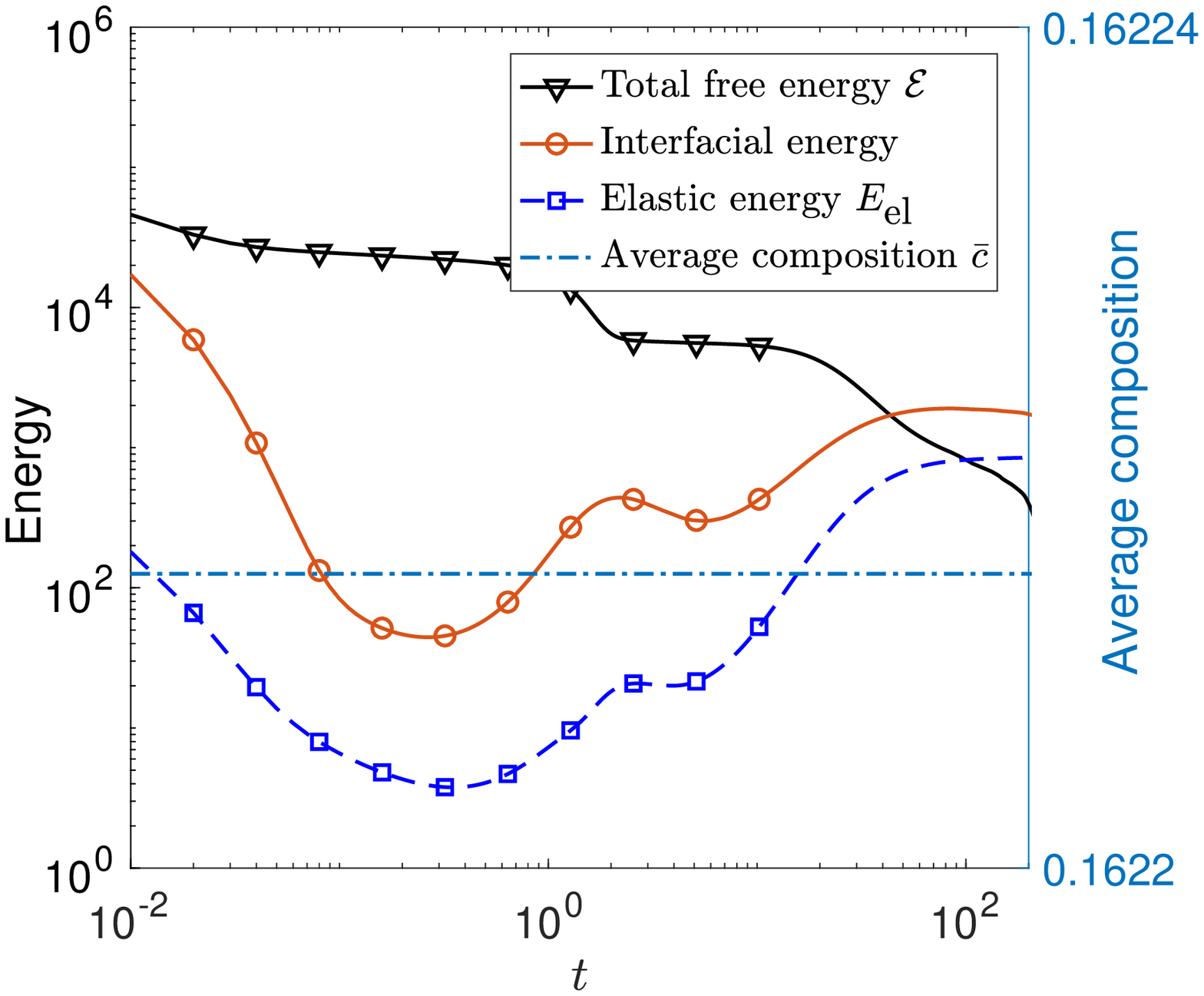}}~
\subfloat[]{\includegraphics[width=0.45\textwidth]{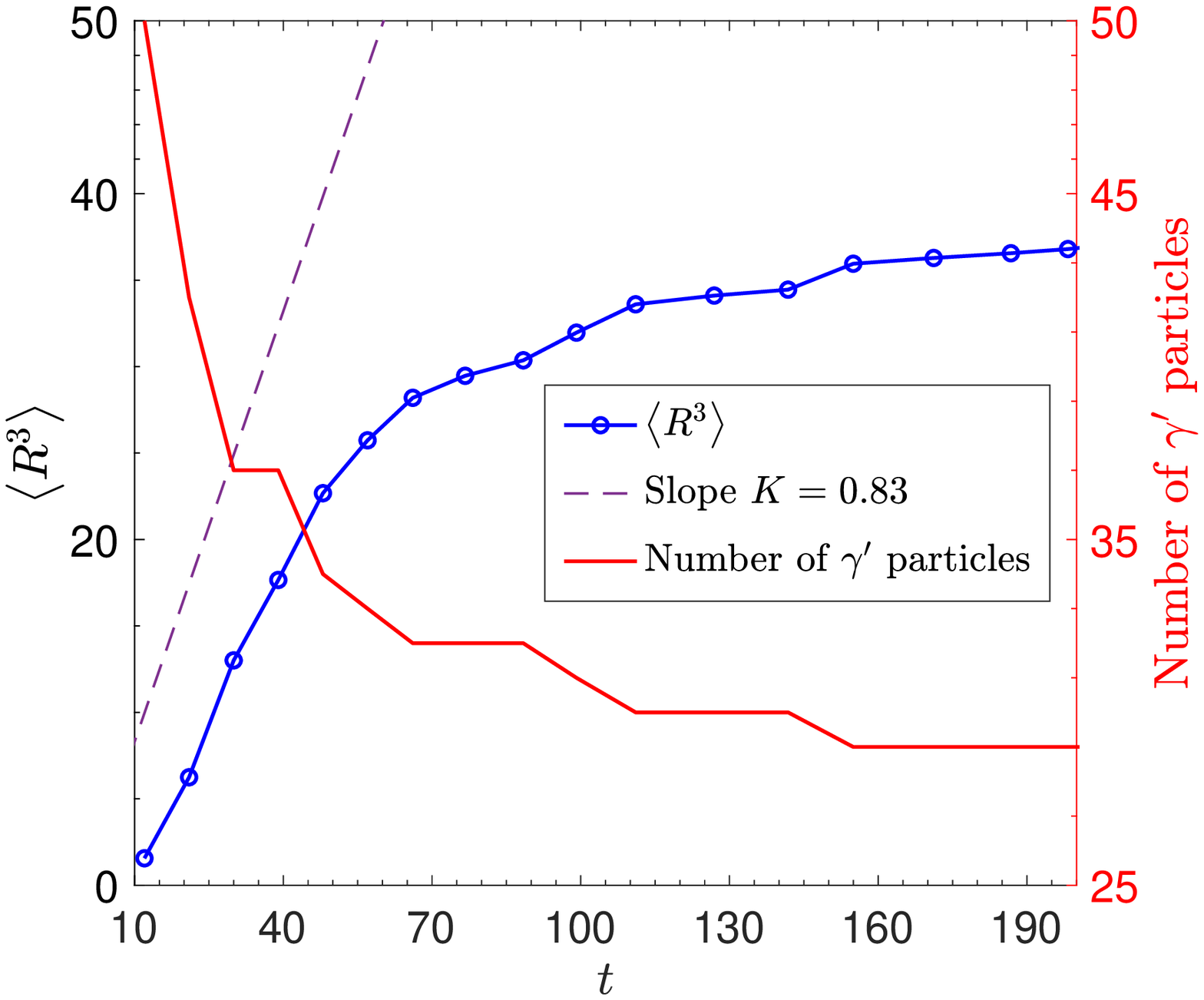}}
\caption{Volume fraction $V_f=17.5\%$. (a) The evolutions of the total free energy, the interfacial energy, the elastic energy (left axis), and the average composition $\bar c$  (right axis).
(b) The evolutions of  $\left<R^3\right>$ and number of $\gamma'$ particles. For visualization convenience, the total free energy is shifted by subtracting $(\tilde{\cal E}^1-4.6\times 10^4)$. }\label{fig:example1.6}
\end{figure}

\begin{figure}
\centering
\subfloat[$t=66$]{\includegraphics[width=0.24\textwidth]{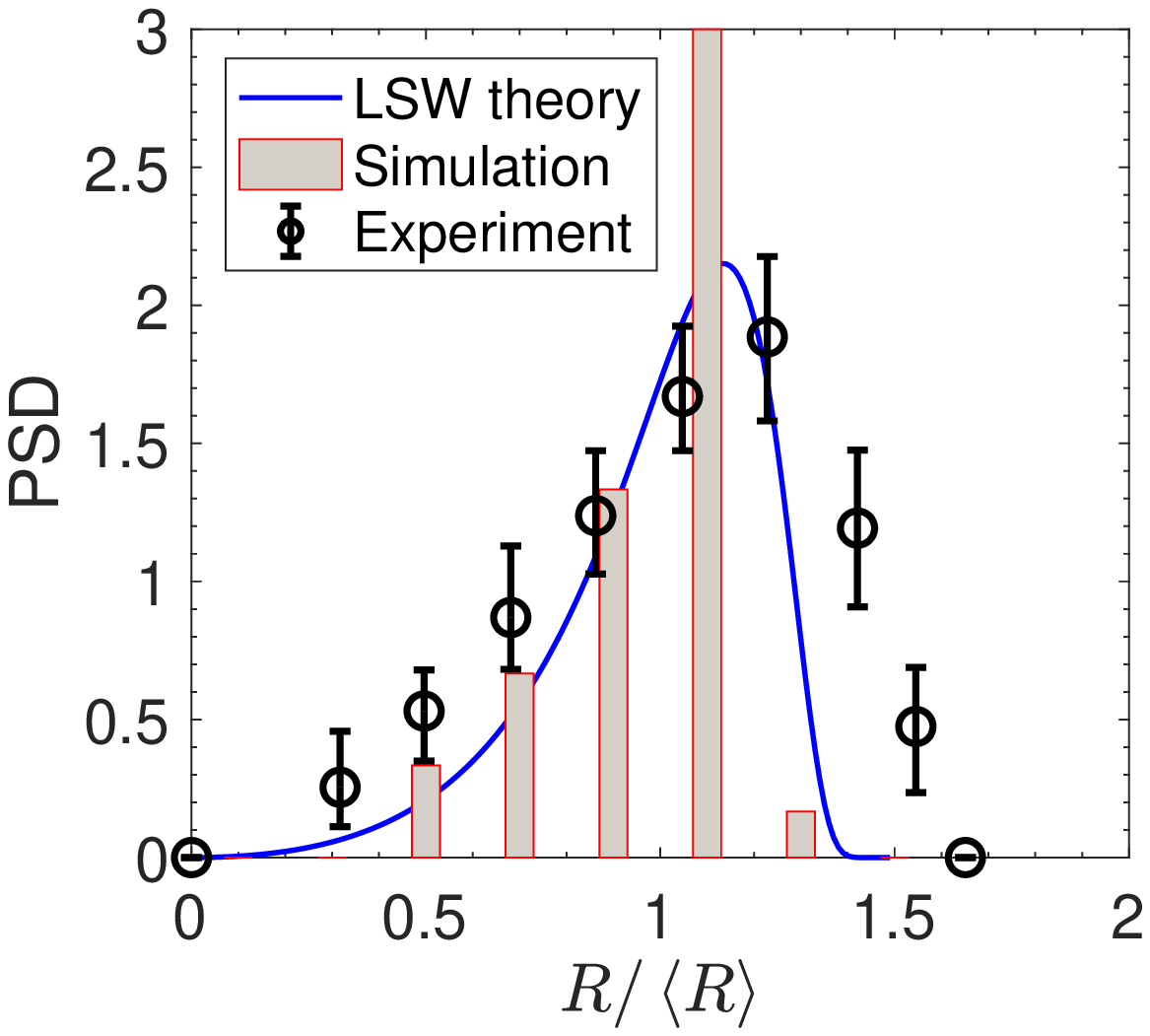}}~
\subfloat[$t=127$]{\includegraphics[width=0.24\textwidth]{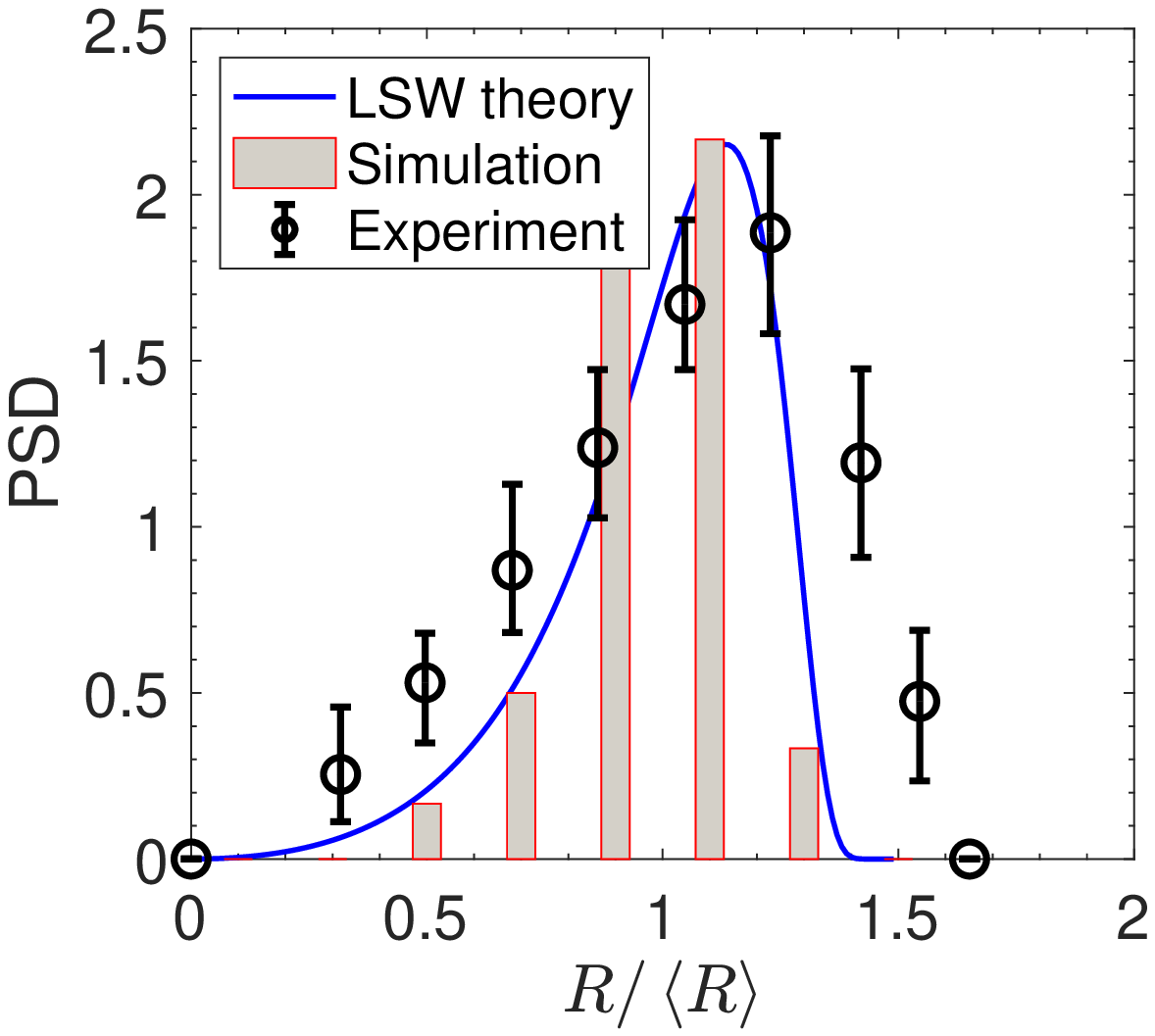}}~
\subfloat[$t=171$]{\includegraphics[width=0.24\textwidth]{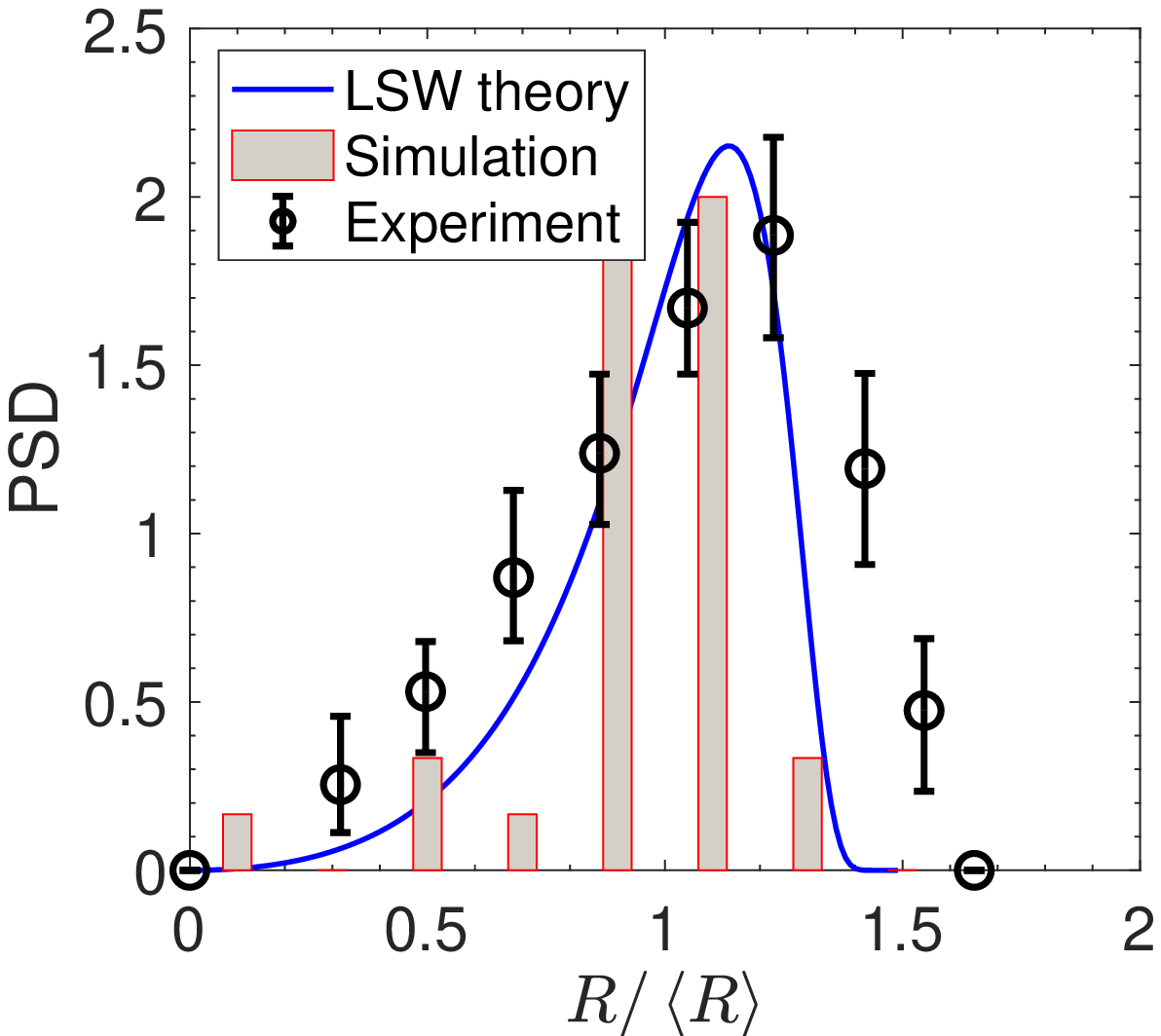}}~
\subfloat[$t=200$]{\includegraphics[width=0.24\textwidth]{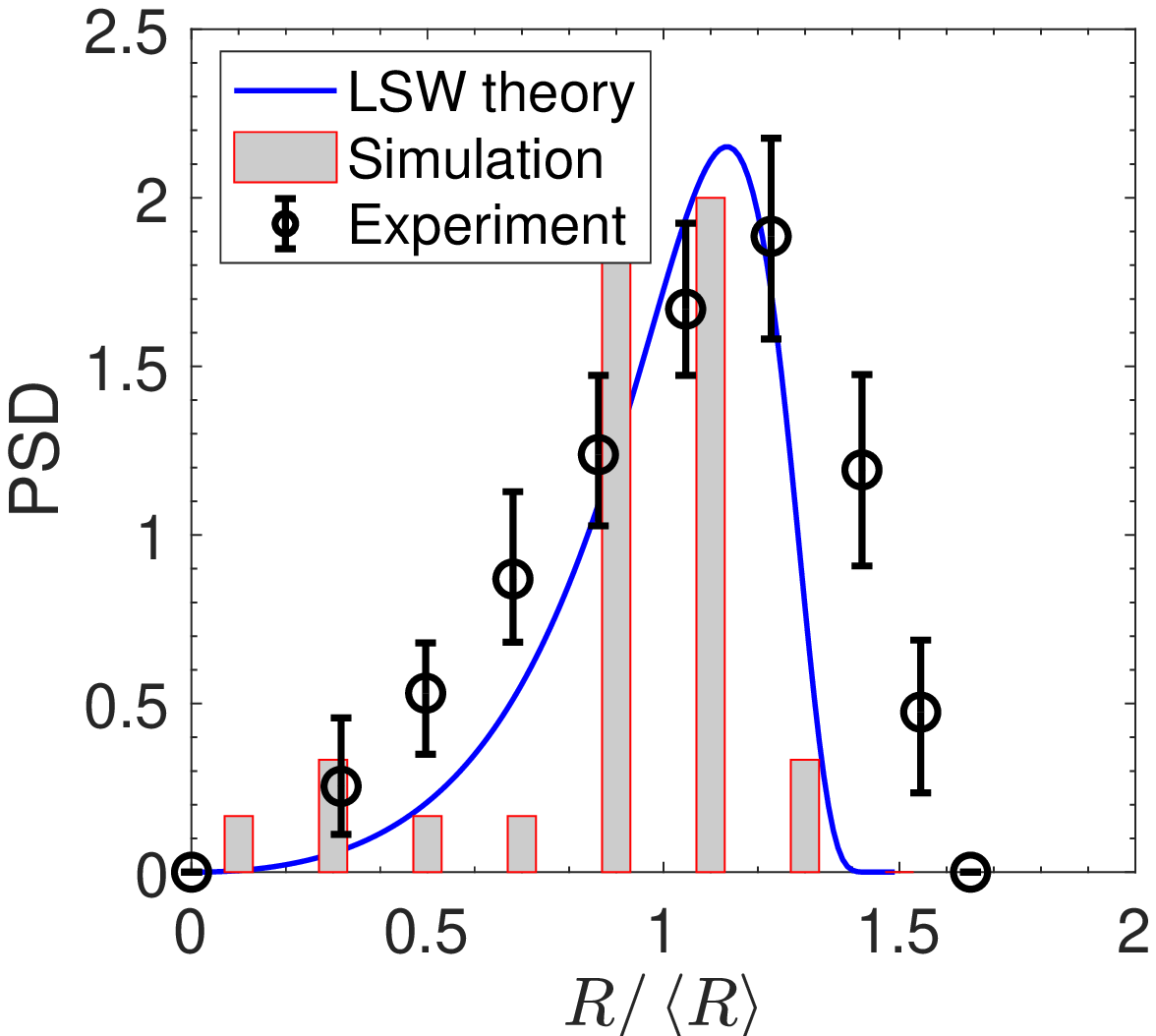}}~
\caption{Volume fraction $V_f=17.5\%$. The temporal evolutions of the normalized PSD of $R/\left<R\right>$ for the $\gamma'$ phase.}\label{fig:example1.7}
\end{figure}

\subsection{Parallel scalability tests}
In this subsection, we study the parallel scalability, in both the weak and strong sense, of the proposed method. 
To begin, we discuss several issues including the influence of different 
overlaps, subdomain solvers and so on. In all simulations, we run the code for the first $10$ time steps
with a fixed  time step size $\Delta t =0.01$ and an $80 \times 80 \times 80$ mesh by using $256$ processor cores. 
To check the influence of the subdomain solvers, we limit the 
test to the classical AS preconditioner and fix the overlapping size to $2\delta=2$. 
The ILU factorizations with $0$, $1$, and $2$ levels of fill-in 
and LU factorization are considered. 
The number of Newton iterations and the averaged number of GMRES iterations together with the total compute times are provided in Table~\ref{tab:subdomain}.
From Table~\ref{tab:subdomain}, we can see the number of Newton iterations is invariable, which means 
the number of Newton iterations is insensitive to the subdomain solver.
In addition, the number of GMRES iterations decreases by increasing the fill-in level, 
but the total compute time keeps growing due to the increased cost of the subdomain 
solver. In summary, we find that the optimal choice in terms of the total compute time is the ILU(0) 
subdomain solver and the reuse strategy can further save nearly 20-30\% of the compute time.

\begin{table}[!htb]
\caption{
Performance of the NKS algorithm with different subdomain solvers.
}
\label{tab:subdomain} 
\centering
\begin{tabular}{|c|ccccc|}
\hline
{Subdomain solver} &ILU(0) &ILU(1) &ILU(2) &LU &ILU(0)-reuse \\
\hline
  Total Newton          					&25                &24                    &25                    &25                     &25  \\
                                     GMRES/Newton      &23.9          &14.5                  &14.9                   &13.5              &24 \\
                                     Total Time (s)           &126.3        &881.4              & 3,815.5           &    8,325.2      &95.1\\
\hline
\end{tabular}
\end{table}

Based on the above tests, we take 
the ILU(0)-reuse as the subdomain solver
and investigate the performance of the NKS solver by changing the type of the AS 
preconditioner and the overlapping factor $\delta$. 
The classical-AS, the left-RAS, and the right-RAS preconditioners with overlapping size $2\delta=0,\, 2,\, 4$ 
 are considered. 
The number of Newton iterations, the averaged number of GMRES iterations, and the total compute 
times are listed in Table \ref{tab:overlap}. From Table \ref{tab:overlap}, we can see that the 
number of Newton iterations is insensitive to the type of the AS preconditioner and the overlapping 
size. We also conclude that the classical-AS preconditioner with $\delta =0$ is superior to the
other AS preconditioners with $\delta =1$ or $2$
  based on the averaged number of GMRES iterations and the compute times.

\begin{table*}[!hbt]
\caption{\upshape
Performance of NKS solver with respect to the types of preconditioner and overlapping sizes. }
\label{tab:overlap} 
\centering
\begin{tabular}{|c|ccc|cc|cc|}
\hline
{Type of preconditioner}
&\multicolumn{3}{|c|}{classical-AS}
&\multicolumn{2}{|c|}{left-RAS}
&\multicolumn{2}{|c|}{right-RAS}\\
\hline
{$\delta$}  &$0$ &$1$ &$2$ &$1$ &$2$   &$1$ &$2$   \\
\hline
\hline
 Total Newton          &26                &25                 &25              &25             &25               &25         &25      \\
GMRES/Newton      &26.2            &24          &19.4           &14.9             &15.9           &14.8         &17.0     \\
Total Time (s)           &69.1         &95.1           &127.0         &79.1         &121.4        &78.7   &192.4 \\
\hline
\end{tabular}
\end{table*}

Then, we focus on parallel scalability tests to examine the weak and strong scaling performance.
Based on the above observations,   we use the classical-AS 
preconditioner with the overlapping $\delta=0$ and employ the ILU(0) factorization 
with the reuse strategy as the subdomain solver in all simulations.
In the weak scalability tests,  a subdomain with 
a fixed number of mesh points  is handled by each processor core. The weak scaling performance is testified by 
increasing the number of the subdomains and the processor cores simultaneously. 
The number of processors increases from $384$ to $3,072$, and the number of physical grid 
points increases 
from $100^3$ to $200^3$. 
 The numbers of Newton and GMRES iterations together 
with the total compute times are provided in Table \ref{tab:ws2d}, which shows
 that  the averaged number of  GMRES iterations and the total 
compute time increase slowly as more processors are used. 
The good weak scalability of our method is validated by the simulation.  
To study the strong scalability of our method, we run the tests with a $250^3$ mesh.
The numbers of nonlinear and linear iterations are reported in Table~\ref{tab:scalability3d}. 
We notice that the number of nonlinear iterations is almost unchanged  
and the average number of linear iterations increases slightly as the number of used processor core increases. 
The total compute time decreases nearly half as the number of processor cores doubles. 
The overall speedup from $384$ to $6,144$ cores is around $6.79$, 
which indicates a good strong parallel efficiency of the proposed algorithm. 

\begin{table}[!htb]
\caption{\upshape
The weak scalability of the proposed method. Here ``NP" denotes the number of processor cores.
}
\label{tab:ws2d} 
\centering
\begin{tabular}{|c|cccc|}
\hline
Mesh size & $100^3$ & $2\times100^3$ & $4\times100^3$ & $200^3$ \\
                                      NP  &384 & 768 & 1,536 & 3,072\\
\hline
 Total Newton                   &26               &26              &26               &26  \\
 GMRES/Newton             &28.0            &37.5             &42.8                &42.1 \\
 Total Time (s)                  &99.8      &115.8       &122.4      &124.4 \\
\hline
\end{tabular}
\end{table}

\begin{table}[!htb]
\caption{
Performance of the NKS solver with different numbers of processor cores. Here ``NP" denotes the number of processor cores.
}
\label{tab:scalability3d} 
\centering
\begin{tabular}{|c|ccccc|}
\hline
 NP & 384&768	 &1,536&3,072 & 6,144  \\
\hline
Total Newton        &25      		 &25        &25        &26      & 26       \\
GMRES/Newton  &40.1		 &39.6     &40.6       &  45.7      &  48.8   \\
Total Time (s) & 	1251.2	 &634.2	&423.3		&236.4& 184.4		\\
Speed up & 	1	 &1.97		&2.96		&5.29 &	6.79	\\
\hline
\noalign{\smallskip}
\end{tabular}
\end{table}

\section{Conclusion}
The Ni-based alloys phase field system consists of CH/GL equations and linear elastic equations, in which the total free energy includes the elastic energy and logarithmic type functionals.
By using the DVD method, we  propose a semi-implicit finite difference scheme, which solves the challenge posed by the special free energy functional. 
The newly presented semi-implicit scheme is proved to be unconditionally energy stable and enjoys the energy dissipative law and mass conservation law. 
Thanks to the  unconditional energy stability of the scheme, based on the desired solution accuracy and the dynamic features of the nucleation, growth, and coarsening procedures, the time step size is adaptively selected  by using an adaptive time stepping strategy. 
At each time step, a large sparse nonlinear algebraic system is constructed by the semi-implicit method.
We introduce a parallel NKS algorithm to efficiently solve the nonlinear system. 
Several three dimensional test cases are used to validate the energy stability of the semi-implicit scheme and study  the morphological evolution of the $\gamma'$ phase through the heat treatment process.
Large scale numerical experiments show that the proposed algorithm can achieve high scalability with up to six thousand processor cores in both strong and weak senses.

\section*{Appendix A}\label{Appendix-A}
In this appendix, we introduce more details on the Ni-based alloys phase field system, i.e.,  the definition of the total free energy \eqref{simple free energy-1}.
According to \cite{ansara1997thermodynamic, liu2018morphology, zhu2002linking},
the $\gamma/\gamma'$ two-phase microstructure in Ni-based alloys can
be described by a composition field variable $c$ of the Al element and  
four long-range order parameter fields $\eta_i:=\eta_i(\mathbf x,t)$ with $i=1,2,3$, and $4$. 
As shown in \cref{lattice}, a  face center cubic (FCC) superlattice 
is divided into four sub-lattices. Each sub-lattice includes four atoms
and the site fractions of Al element in sub-lattices A, B, C, and D are denoted as $c\eta_i$.
The overall composition $c$ can be computed by $c=\sum\limits_{i=1}^4 c\eta_i/4$,
which implies constraints $\sum\limits_{i=1}^4 \eta_i=4$ and $\eta_i\in[0,4]$. 
The composition field variable for the Ni element is set as $1-c$.
For Ni-based alloys including  $\gamma$/$\gamma'$ phases, the composition field variable $c$ 
is in the range of $(0,0.25)$. 
In the disordered phase  ($\gamma$ phase),  the four long-range order parameter fields $\eta_i$ are the same, i.e., $\eta_i=1$. 
If the first three  long-range order parameter fields
have the same value and different from the fourth, there is 
the  ordering phase ($\gamma'$ phase).
The total free energy of a $\gamma/\gamma'$ two-phase system, including the local
chemical (Gibbs) free energy, the interfacial
energy, and the elastic energy, can be written as the following function of $c$ and  $\eta_i$ with $i=1,2,3$
\cite{ansara1997thermodynamic, liu2018morphology, zhu2002linking}
\begin{equation}
\label{free energy}
{\cal E}=\int\limits_\Omega\left[\frac1{V_m}E_{\textnormal G}(c,\eta_1,\eta_2,\eta_3)+\frac{ \gamma_c}2| \nabla c|^2+\frac{\gamma_{\eta}}{2}
\sum\limits_{i=1}^3|\nabla \eta_i|^2+E_{\textnormal{el}}(c)\right]\hbox{d} \mathbf x,
\end{equation}
where $\Omega\in\mathbb{R}^d$ is the computational domain, $V_m=1.48\times 10^{-5}{\hbox{m}}^3/\hbox{mol}$ is the mole volume,
$E_{\textnormal G}(c,\eta_1,\eta_2,\eta_3)$ is the local
chemical free energy, and $E_{\textnormal{el}}(c)$ is the elastic energy. Here, $\gamma_c=2.5\times 10^{-9}$J/m  and $\gamma_{\eta}=6.0\times 10^{-12}$J/m 
are the gradient energy coefficients of the composition and long-range order
parameters, respectively.

\begin{figure}
\centering
{\includegraphics[width=0.5\textwidth]{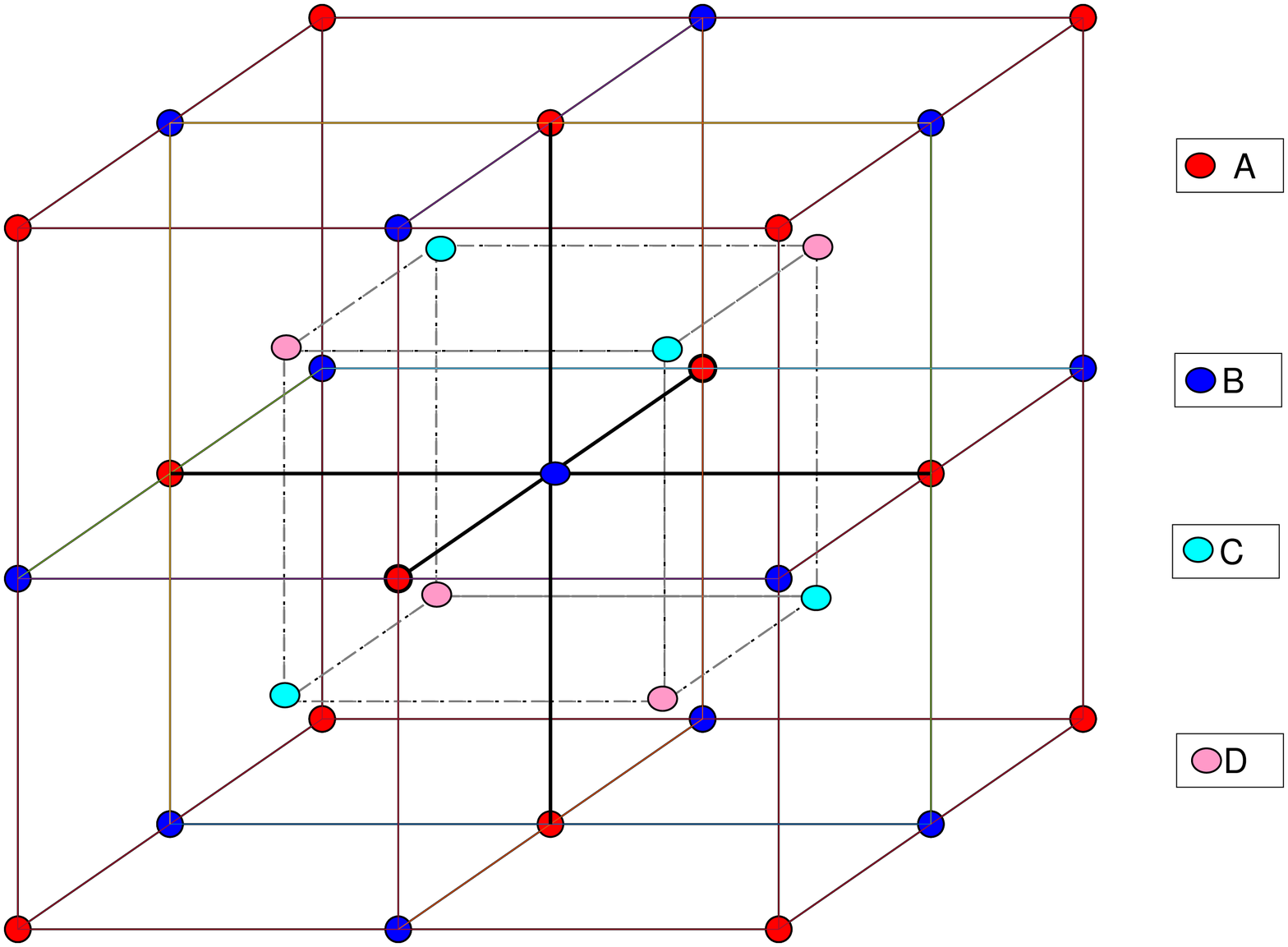}}
\caption{A FCC  lattice is divided into four cubic sub-lattices, labeled A, B, C, and D, each with lattice constant $a$. 
The D site is nonmagnetic and the others are magnetic.}\label{lattice}
\end{figure}

The Gibbs free energy defines the basic thermodynamic
properties of the Ni-based alloys system, which can be written in terms of the composition
and the long-range order parameter fields as \cite{ansara1997thermodynamic, zhu2002linking}
\begin{equation}
E_{\textnormal G}(c,\eta_1,\eta_2,\eta_3)=E^{\textnormal{dis}}(c)+E^{\textnormal{ord}}(c,\eta_1,\eta_2,\eta_3)-E^{\textnormal{ord}}(c,1,1,1),
\end{equation}
where $E^{\textnormal{dis}}$ and $E^{\textnormal{ord}}$ are the Gibbs free energies of the
disordered phase and ordered phase, respectively.
The Gibbs free energy  of a pure disorder phase is defined as
\begin{equation}
E^{\textnormal{dis}}(c)=E^{\textnormal{dis}}_0(c)+\Delta E^{\textnormal{dis}}_{\textnormal{ideal}}(c)+\Delta E^{\textnormal{dis}}_{\textnormal{ex}}(c),
\end{equation}
where the reference Gibbs energy $E^{\textnormal{dis}}_0(c)=E_0^{\textnormal{Al}}c+E_0^{\textnormal{Ni}}(1-c)$, 
the entropy $\Delta E^{\textnormal{dis}}_{\textnormal{ideal}}(c)=RT\Psi(c)$, 
and the excess Gibbs energy  $\Delta E^{\textnormal{dis}}_{\textnormal{ex}}(c)=c(1-c)\sum\limits_{i=0}^3[L_i(2c-1)^i ]$.
The Gibbs free energy $E^{\textnormal{ord}}$ of a pure order phase  is given by
\begin{equation}
E^{\textnormal{ord}}(c,\eta_1,\eta_2,\eta_3)=E^{\textnormal{ord}}_0(c,\eta_1,\eta_2,\eta_3)+\Delta E^{\textnormal{ord}}_{\textnormal{ideal}}(c,\eta_1,\eta_2,\eta_3)+\Delta E^{\textnormal{ord}}_{\textnormal{ex}}(c,\eta_1,\eta_2,\eta_3),
\end{equation}
where the reference Gibbs energy $E^{\textnormal{ord}}_0$, 
the entropy  $\Delta E^{\textnormal{ord}}_{\textnormal{ideal}}$, 
and the excess Gibbs energy  $\Delta E^{\textnormal{ord}}_{\textnormal{ex}}(c)$ are respectively defined as
\begin{equation}
E^{\textnormal{ord}}_0(c,\eta_1,\eta_2,\eta_3)=6U_1c^2\sum_{i=1}^3\phi_i^2,
\end{equation}
\begin{equation}
\Delta E^{\textnormal{ord}}_{\textnormal{ideal}}(c,\eta_1,\eta_2,\eta_3)=\frac{RT}4\sum_{i=1}^4\Psi(c\eta_i),
\end{equation}
and
\begin{equation}
\Delta E^{\textnormal{ord}}_{\textnormal{ex}}(c,\eta_1,\eta_2,\eta_3)=-2U_1c^2\sum_{i=1}^3\phi_i^2+12U_4(1-2c)c^2\sum_{i=1}^3\phi_i^2-48U_4c^3\phi_1\phi_2\phi_3.
\end{equation}
Here $\phi_1=1-\frac{\eta_1+\eta_2}2$, $\phi_2=1-\frac{\eta_1+\eta_3}2$, and $\phi_3=1-\frac{\eta_2+\eta_3}2$.
All the parameters  in the above free energies such as $E_0^{\textnormal{Al}}$, $E_0^{\textnormal{Ni}}$, $L_i$ with $i=0,1,2,3$,
$U_1$, and $U_4$ are functions of temperature $T$ and the  gas
constant $R$, which can be
obtained from CALPHAD database \cite{ansara1997thermodynamic, saunders1998calphad}.

Since the three long-range order parameter fields $\{\eta_i\}_{i=1}^3$ take the same value in  the $\gamma$ and $\gamma'$ phase,
the phase field system can be approximately described by a composition field variable $c$ and one long-range order parameter $\eta$
by taken $\eta(\mathbf x,t)=\eta_1=\eta_2=\eta_3$ \cite{liu2018morphology}.
The total free energy is then rewritten as 
\begin{equation}
\label{simple free energy}
{\cal E}=\int\limits_\Omega\left[\frac 1 {V_m}E_{\textnormal G}(c,\eta)+\frac{ \gamma_c}2| \nabla c|^2+\frac{3\gamma_{\eta}}{2}
|\nabla \eta|^2+E_{\textnormal {el}}(c)\right]\hbox{d} \mathbf x.
\end{equation}
The contour plot of the Gibbs free energy $E_{\textnormal G}(c,\eta)$ at $T=1,037$ K is plotted in \cref{fig:freeenergy}-(a). There are two wells on the surface and  they correspond to the 
equilibrium states of the disorder and order phases. 
The values of the composition and the order parameter for these two phases are
$\eta =1, c=0.137$, and $\eta=0.02, c=0.229$, respectively.
 \cref{fig:freeenergy}-(b)
 displays the relative chemical free energy of the $\gamma$ phase and $\gamma'$ phase of the Ni-based alloys as a function of composition $c$.
  In the stress-free state, the  maximum  driving  force is approximately 292.5 J/mol. 
  \cref{fig:freeenergy}-(c) shows the phase diagram for $\gamma$/$\gamma'$ phase with $T\in[800,\ 1,300]$, which is consistent with the experiment results \cite{massalski1989phase}
  and calculated results by using EAM potential \cite{mishin2004atomistic}.

\begin{figure}[H]
\centering
\subfloat[]{
\includegraphics[width=0.32\textwidth]{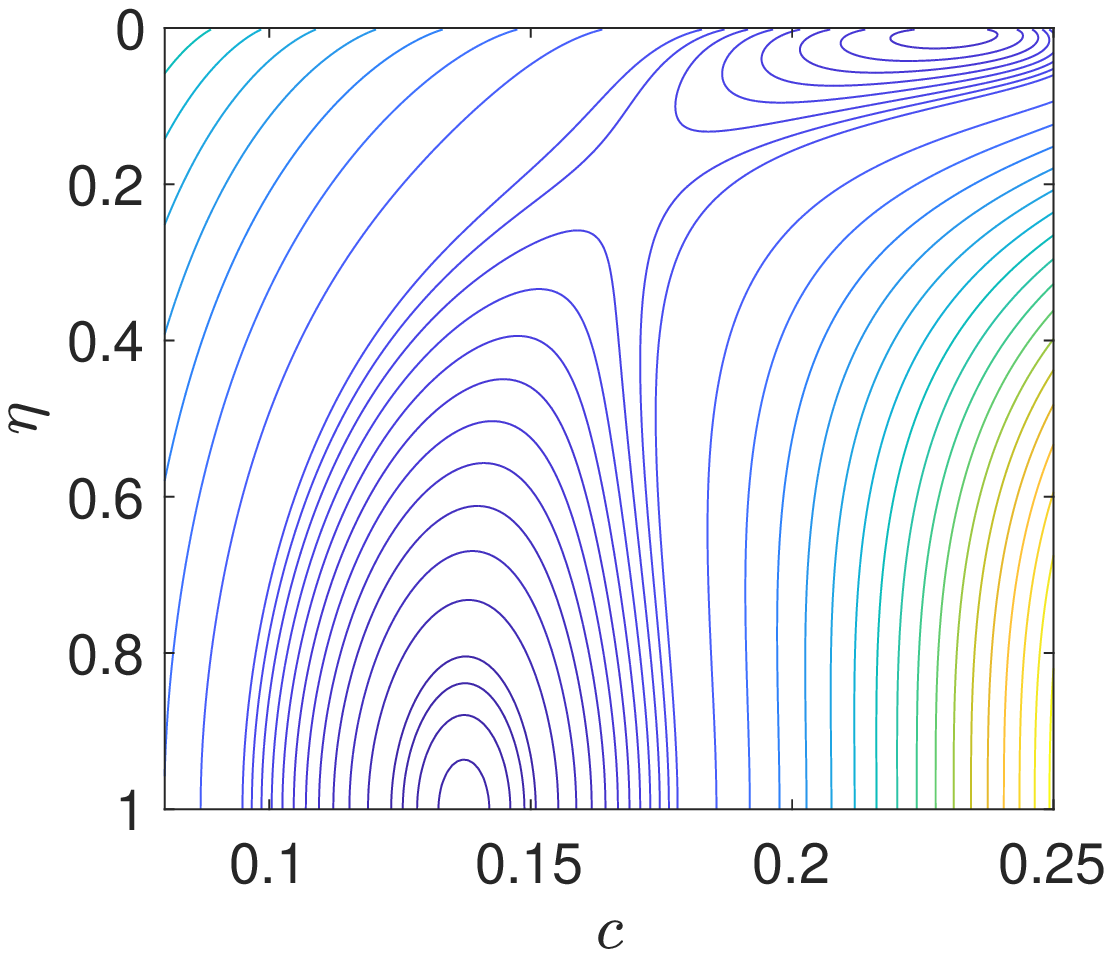}}
\subfloat[]{\includegraphics[width=0.32\textwidth]{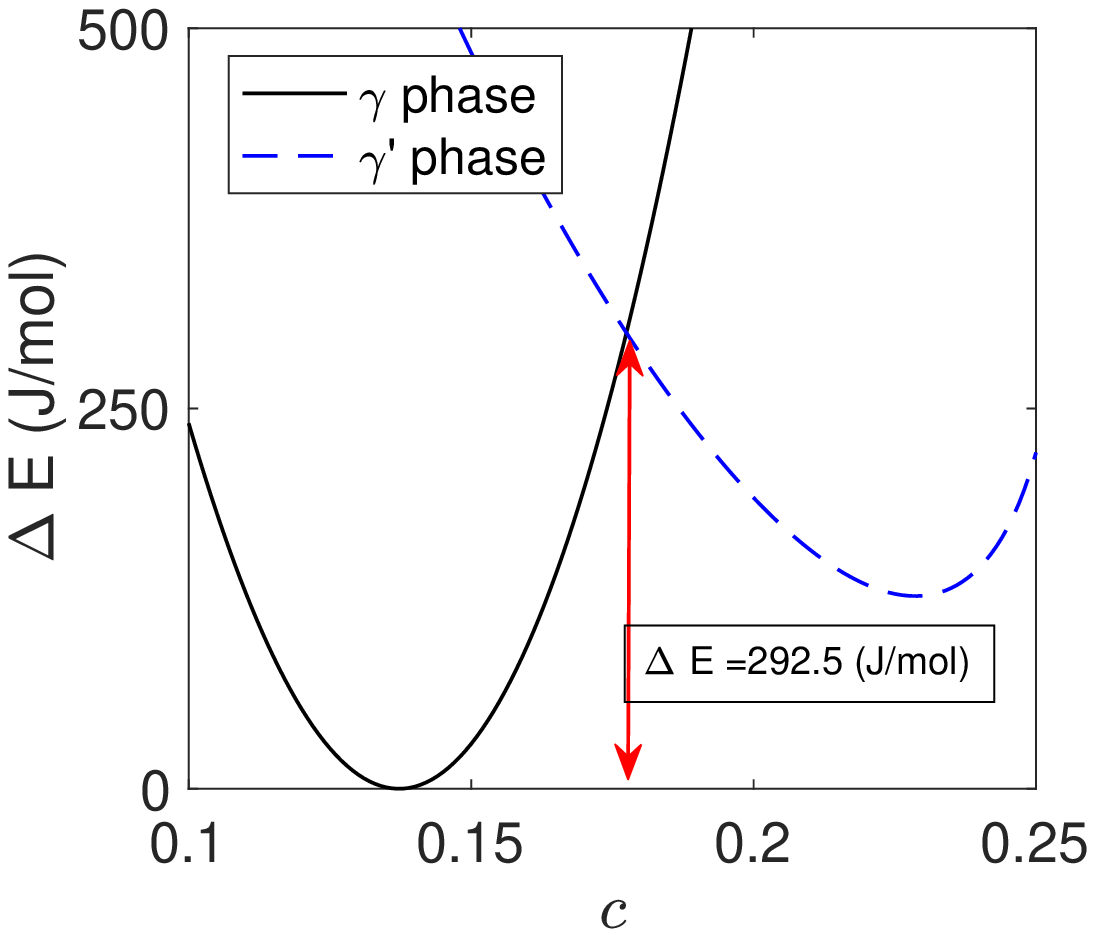}}
\subfloat[]{\includegraphics[width=0.32\textwidth]{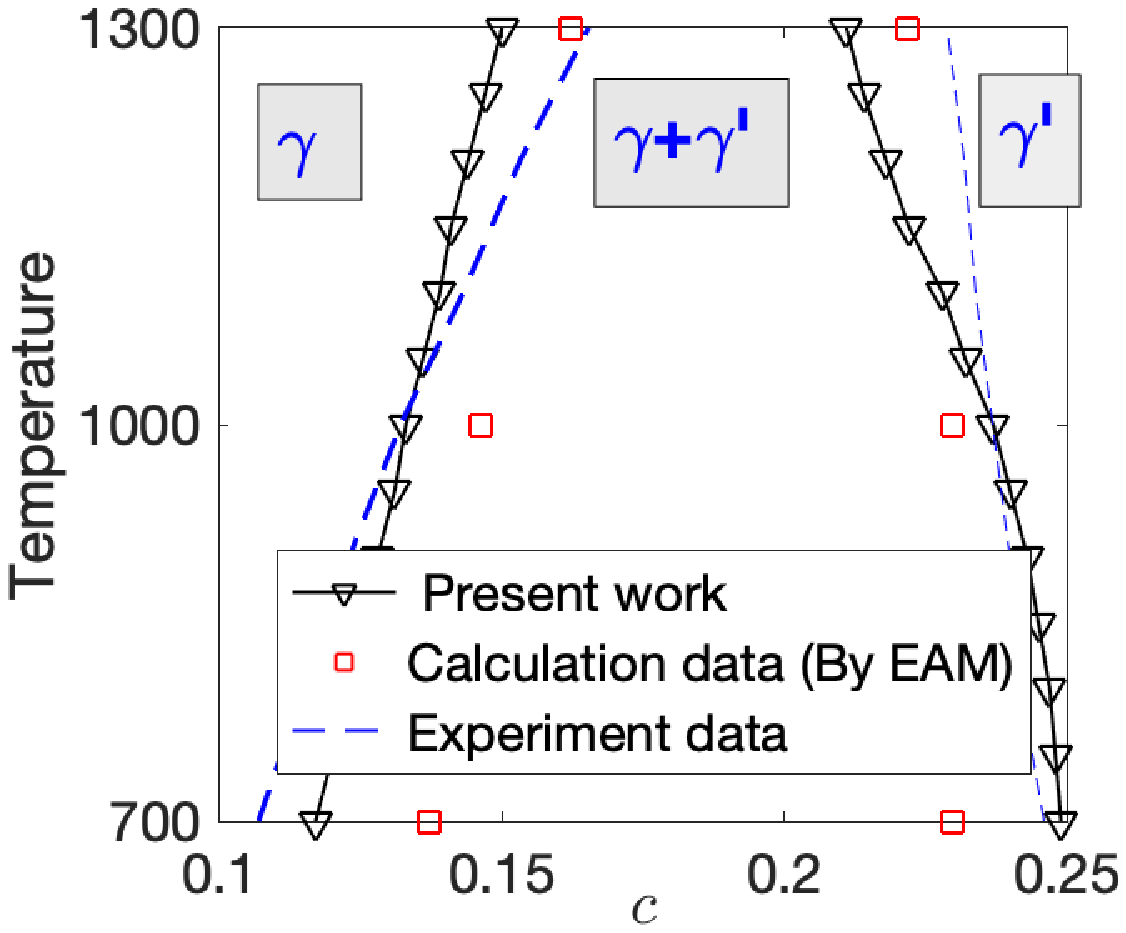}}
\caption{(a) Contour plot of the Gibbs free energy $E(c,\eta)$ at $T=1,073$ K. (b) The Gibbs free energy of the $\gamma$ and $\gamma'$ phases. (c) Phase boundary of the
$\gamma$ and $\gamma'$ phases calculated by $E(c,\eta)$ for temperatures ranging from 800K to 1300K.}
\label{fig:freeenergy}
\end{figure}

The elastic energy has important contributions to the morphological
evolution and coarsening kinetics of $\gamma'$ particles, which can be
obtained by using the micro-elasticity theory \cite{khachaturyan1983theory, turchi2000phase}.
The main factors of the elastic energy fluctuation in the computational domain are: (1) the lattice mismatch between the different phases; (2) the
elastic constant inhomogeneity; (3) externally applied stress.
In this paper, the phase field system is constructed without external stress.
According to \cite{zhu2004three}, 
the elastic constant difference
between the $\gamma$ and $\gamma'$ phases (usually less than 7\%) can be ignored, if no external stress is applied.
Therefore, the elastic strain energy   \cite{hu2001phase,li2012phase} is approximately calculated by:
\begin{equation}
E_{\textnormal{el}}(c)=\frac 12 \epsilon^{\textnormal{el}}: C:\epsilon^{\textnormal{el}},
\end{equation}
where $C$ is the Hooke's tensor, $\epsilon^{\textnormal{el}}=(\epsilon^{\textnormal{el}}_{IJ})$ is the elastic strain.
The elastic strain is reformulated as $\epsilon^{\textnormal{el}}=\epsilon-\epsilon^0$,
where $\epsilon$ is the total strain related to the displacement and the eigenstrain $\epsilon^0$  expressed as
$\epsilon^0=\epsilon_0(c-\bar{c}) {\cal I}$.
Here $\epsilon_0=\frac 1 a \frac{\hbox{d} a}{\hbox{d}c}\approx \frac{a_{\gamma'}-a_\gamma}{a_\gamma(c_{\gamma'}-c_\gamma)}=0.049$ is the composition expansion coefficient of the lattice parameter,
${\cal I}$ is the identity matrix, and  $\bar{c}=\frac{1}{|\Omega|}\int_\Omega c(\bx,t)\md \bx$ is the average composition.
The lattice constants for the $\gamma$ and $\gamma'$ phases are denoted as $a_{\gamma}$ and $a_{\gamma'}$, respectively.
The Ni-based alloys can be taken as a hexagonal crystal, which has three independent elastic constants $C_{11},\,C_{12},$ and $ C_{44}$. 
The elastic strain energy is finally written as 
\begin{equation}
\label{elastic-energy}
\begin{aligned}
E_{\textnormal{el}}(c):
=\frac 12 \left(C_{12}\left(\sum\limits_{I=1}^d\epsilon_{II}^{\textnormal{el}}\right)^2+(C_{11}-C_{12})\sum\limits_{I=1}^d\left(\epsilon_{II}^{\textnormal{el}}\right)^2+
4C_{44}\sum\limits_{I<J}(\epsilon_{IJ}^{\textnormal{el}})^2\right).
\end{aligned}
\end{equation}

The morphological evolution of the  composition field variable $c(\mathbf{x},t)$
and long-range order parameter $\eta(\mathbf{x},t)$
are described by the CH--GL coupled equations:
\begin{equation}\label{CHGL-1}
\left\{
\begin{aligned}
&\frac{\partial c}{\partial t}=\nabla \left[M_c\nabla \frac {\delta {\cal E}}{\delta c}\right] \quad \hbox{in}~\Omega\times(0,\mathcal {T}],\\
& \frac{\partial \eta}{\partial t}=-M_\eta\frac{\delta {\cal E}}{\delta \eta}\quad \hbox{in}~\Omega\times(0,\mathcal {T}],
\end{aligned}
\right.
\end{equation}
where $M_c$ and $M_\eta$ are two mobility parameters. 
By taken dimensionless units for time $t$ as $M_\eta\Delta E$ with $\Delta E=3.3 \times 10^7$ $ \hbox{J/m}^3$, for $\bx$ as $1/{\mathbb L}$
with ${\mathbb L}=1.5$ nm,
we obtain the dimensionless equations \cref{CHGL}, in which $\varkappa=0.008$.


\section*{Appendix B}\label{Appendix-B}
In this appendix, we present the definition of $\tilde{\mathbf{G}}^1_{i}$. 
In the definition of  the variational derivatives \cref{variational derivatives}, $\frac{\partial E_{\textnormal G}(c,\eta)}{\partial c}$, $\frac{\partial E_{\textnormal G}(c,\eta)}{\partial \eta}$, and $\frac{\partial E_{\textnormal{el}}(c)}{\partial c}$
are given as  follows
\begin{equation}
\left\{
\begin{aligned}
&\frac{\partial E_{\textnormal G}(c,\eta)}{\partial c}=M_0+M_1(\phi)c+M_2(\phi)c^2+M_3c^3+M_4c^4+\frac{RT}4\varphi'_c(c),\\
&\frac{\partial E_{\textnormal G}(c,\eta)}{\partial \eta}=M_5(c)\phi+M_6(c)\phi^2
+\frac{RT}4\varphi'_\eta,\\
&\frac{\partial E_{\textnormal{el}}(c)}{\partial c}=-\epsilon_{0}{\cal I}:C:\epsilon^{\textnormal{el}},
\end{aligned}
\right.
\end{equation}
where the dimensionless  parameters (functions) are defined by
\begin{equation}
\left\{
\begin{aligned}
&M_0=\frac{1}{V_m\Delta E}\left[E_0^{\textnormal{Al}}-E_0^{\textnormal{Ni}}+(L_0-L1+L2-L3)\right],
\\
&M_1(\phi)=\frac{1}{V_m\Delta E}\left[-(2L_0 - 6L_1 + 10L_2 - 14L_3)+24U_1\phi^2+72U_4\phi^2\right],
\\
&M_2(\phi)=\frac{1}{V_m\Delta E}\left[-(6L_1 - 24L_2 + 54L_3)-216U_4\phi^2-144U_4\phi^3\right],\\
&M_3=-\frac{16L_2 - 80L_3}{V_m\Delta E},\\
&M_4=-\frac{40L_3}{V_m\Delta E},\\
&M_5(c)=-\frac{24U_1c^2+72U_4(1-2c)c^2}{V_m\Delta E},\\
&M_6(c)=\frac{144U_4c^3}{V_m\Delta E},\\
&\varphi'_c=\frac{1}{V_m\Delta E}\left[3\eta\Phi'(c\eta)+(4-3\eta)\Phi'(c(4-3\eta))\right],\\
&\varphi'_\eta=\frac{1}{V_m\Delta E}\left[3c\Phi'(c\eta)-3c\Phi'(c(4-3\eta))\right].
\end{aligned}
\right.
\end{equation}
The definition of  $\tilde{\mathbf{G}}^1_{i}$ is then given by
\begin{equation}
\tilde{\mathbf{G}}^1_{i}=\left(
\begin{aligned}
&\frac{M_1(\tilde{\phi}_i^n)+M_1(\tilde{\phi}_i^{n+1})}2\tilde c_{i}^{n+1/2}+\frac{M_2(\tilde{\phi}_i^n)+M_2(\tilde{\phi}_i^{n+1})}6\frac{(\tilde c_{i}^{n+1})^3-(\tilde c_{i}^{n})^3}{\tilde c_{i}^{n+1}-\tilde c_{i}^{n}}\\
&\qquad\qquad+\frac {M_3}4 \frac{(\tilde c_{i}^{n+1})^4-(\tilde c_{i}^{n})^4}{\tilde c_{i}^{n+1}-\tilde c_{i}^{n}} +\frac {M_4}5 \frac{(\tilde c_{i}^{n+1})^5-(\tilde c_{i}^{n})^5}{\tilde c_{i}^{n+1}-\tilde c_{i}^{n}} +M_0,\\
&\frac{M_5(\tilde{c}_i^n)+M_5(\tilde{c}_i^{n+1})}2\tilde \phi_{i}^{n+1/2}+\frac{M_6(\tilde{c}_i^n)+M_6(\tilde{c}_i^{n+1})}6\frac{(\tilde \phi_{i}^{n+1})^3-(\tilde \phi_{i}^{n})^3}{\tilde \phi_{i}^{n+1}-\tilde \phi_{i}^{n}}
\end{aligned}
\right).
\end{equation}


\end{document}